\documentclass[]{scrartcl}

\usepackage[utf8]{inputenc}
\usepackage[USenglish]{babel}
\usepackage{csquotes}
\usepackage[normalem]{ulem} 

\usepackage[a4paper,top=27mm,bottom=20mm,inner=25mm,outer=20mm]{geometry}

\usepackage[%
  backend=bibtex,bibencoding=ascii,
  style=numeric-comp,
  giveninits=true, uniquename=init, 
  natbib=true,
  url=false,
  doi=true,
  isbn=false,
  backref=false,
  maxnames=99,
  ]{biblatex}
\usepackage{amsmath}
\allowdisplaybreaks
\numberwithin{equation}{section}
\usepackage{amssymb}
\usepackage{commath}
\usepackage{mathtools}
\usepackage{bbm}
\usepackage{nicefrac}
\usepackage{subdepth}

\usepackage{algorithm}
\usepackage{algpseudocode}

\usepackage{tikz}

\usepackage{siunitx}
\usepackage{amsthm}
\usepackage{thmtools}
\usepackage{etoolbox}
\usepackage{multirow}

\usepackage{breakurl}

\makeatletter
\patchcmd{\thmt@setheadstyle}
 {\bgroup\thmt@space}
 {\thmt@space}
 {}{}
\patchcmd{\thmt@setheadstyle}
 {\egroup\fi}
 {\fi}
 {}{}
\makeatother
\declaretheoremstyle[
  bodyfont=\normalfont\itshape,
  headformat=\NAME\ \NUMBER\NOTE,
]{myplain}
\declaretheoremstyle[
  headformat=\NAME\ \NUMBER\NOTE,
]{mydefinition}
\newcommand{\envqed}{{\lower-0.3ex\hbox{$\triangleleft$}}}
\declaretheorem[style=myplain,numberwithin=section]{theorem}
\declaretheorem[style=myplain,numberlike=theorem]{lemma}

\declaretheorem[style=myplain,numberlike=theorem]{corollary}
\declaretheorem[style=mydefinition,numberlike=theorem,qed=\envqed]{definition}
\declaretheorem[style=mydefinition,numberlike=theorem,qed=\envqed]{remark}

\usepackage[plainpages=false,pdfpagelabels,hidelinks,unicode]{hyperref}

\usepackage{color}
\usepackage{graphicx}
\usepackage[small]{caption}
\usepackage{subcaption}

\begingroup\expandafter\expandafter\expandafter\endgroup
\expandafter\ifx\csname pdfsuppresswarningpagegroup\endcsname\relax
\else
\fi

\usepackage{booktabs}
\usepackage{rotating}
\usepackage{multirow}

\usepackage{enumitem}

\usepackage{ifluatex}
\ifluatex
  \usepackage[no-math]{fontspec}
\else
  \usepackage[T1]{fontenc}
\fi
\usepackage{newpxtext,newpxmath}
\usepackage{comment}

\usepackage{xparse}

\let\epsilon\varepsilon
\let\phi\varphi
\let\rho\varrho

\usepackage{xspace}

\providecommand\R{}
\renewcommand{\R}{\mathbb{R}}

\renewcommand{\vec}[1]{\underline{#1}}

\newcommand{\I}{\operatorname{I}}

\newcommand{\mean}[1]{{\{\mkern-6mu\{}#1{\}\mkern-6mu\}}}

\newcommand{\rhoj}{\begin{pmatrix}
\rho\\
\tg
    \end{pmatrix}}
\newcommand{\vel}{\nu}
\newcommand{\z}{z}
\newcommand{\Drho}{D^{\rho}}
\newcommand{\Dg}{D^{\tg}}
\newcommand{\ds}[1]{d_{[#1]}}

\renewcommand{\mean}[1]{\left\langle #1 \right\rangle}

\newcommand{\tg}{\tilde{g}}

\newcommand{\orcid}[1]{ORCID:~\href{https://orcid.org/#1}{#1}}
\usepackage{authblk}

\newenvironment{keywords}{\par\textbf{Key words.}}{\par}
\newenvironment{AMS}{\par\textbf{AMS subject classification.}}{\par}

\title{Domain-of-dependence-stabilized cut-cell discretizations of linear
kinetic models with summation-by-parts properties
}

\author[1]{Louis~Petri\thanks{\orcid{0009-0005-5322-3982}}}
\author[2]{Sigrun~Ortleb\thanks{\orcid{0000-0003-1341-853X}}}
\author[3]{Gunnar~Birke\thanks{\orcid{0000-0002-2008-6679}}}
\author[3]{Christian~Engwer\thanks{\orcid{0000-0002-6041-8228}}}
\author[1]{Hendrik~Ranocha\thanks{\orcid{0000-0002-3456-2277}}}

\affil[1]{Institute of Mathematics, Johannes Gutenberg University Mainz, Staudingerweg 9, 55128 Mainz, Germany}
\affil[2]{Department of Mathematics, RWTH Aachen University, Schinkelstraße 2, 52062 Aachen, Germany}
\affil[3]{Applied Mathematics, University of M\"unster, Orl\'eans-Ring 10, 48149 M\"unster, Germany}

\date{January 9, 2026} 

\makeatletter
\hypersetup{pdfauthor={Louis Petri, Gunnar Birke, Christian Engwer, Hendrik Ranocha}} 
\hypersetup{pdftitle={DoDKineticEqs
}} 
\makeatother

\begin{document}

\maketitle

\begin{abstract}
\noindent
  We employ the summation-by-parts~(SBP) framework to extend
the recent domain-of-dependence~(DoD) stabilization for cut cells to linear
kinetic models in diffusion scaling. Numerical methods for these models are
challenged by increased stiffness for small scaling parameters and the
necessity of asymptotics preservation regarding a parabolic limit equation.
As a prototype model, we consider the telegraph equation in one spatial
dimension subject to periodic boundary conditions with an asymptotic limit
given by the linear heat equation.
We provide a general semidiscrete stability result for this model when
spatially discretized by arbitrary periodic (upwind) SBP operators and
formally prove that the fully discrete scheme is asymptotic preserving.
Moreover, we prove that DoD with central numerical
fluxes leads to periodic SBP operators. Furthermore, we show that adapting
the upwind DoD scheme yields periodic upwind SBP operators. Consequently,
the DoD stabilization possesses the desired properties considered in the first part
of this work and thus leads to a stable and asymptotic preserving scheme for
the telegraph equation. We back our theoretical results with numerical
simulations and demonstrate the applicability of this cut-cell
stabilization for implicit time integration in the heat equation limit.

\end{abstract}

\begin{keywords}
  asymptotic-preserving schemes,
  cut-cell meshes,
  discontinuous Galerkin methods,
  domain-of-dependence stabilization,
  telegraph equation,
  summation-by-parts operators,
  small-cell-problem
\end{keywords}

\begin{AMS}
  65M06, 
  65M12,  
  65M20, 
  65M60, 
  65M70  
\end{AMS}

\section{Introduction}
In many applications of hyperbolic partial differential equations, the complexity
of the underlying geometry challenges the human and
computational resources required to construct a body-fitted mesh. Cut-cell methods
avoid these tasks by simply cutting along the geometry boundaries out of a
Cartesian background mesh. Of course, one has to deal with different problems by
following this strategy, the most obvious being the irregularity and sizes of the
cut cells. Coupling a spatial discretization on such a mesh with an explicit
time integration method, the latter issue results in the well-known \textit{small-cell}
problem which potentially requires arbitrarily small time steps
to obtain a stable fully-discrete scheme. Therefore, one desires a special treatment
of these small cut cells because of their potential irregularity and cell size.

Cut-cell methods for parabolic differential equations are widely explored (e.g.,
\cite{bastian2009unfitted, barret1987, burmann2010ghost, hansbo2002}).
While the development for hyperbolic finite-volume cut-cell methods is a topic of research
since some decades, e.g., with the h-box method \cite{berger2003hbox, berger2012hbox},
stabilization techniques for discontinuous Galerkin (DG) methods gained popularity over
the last years.
Variants of the ghost penalty method,
originally designed for elliptic cut-cell problems, are examined in
\cite{Sticko2019,Guerkan2020,Fu2021}.
Recently, the state redistribution method has been extended from
finite volume schemes \cite{Giuliani2021} to DG methods \cite{BERGER2021,TAYLOR2025}.
Here, we consider the domain-of-dependence (DoD) stabilization, that was introduced in
\cite{engwer2020stabilized}.
DoD addresses the DG semidiscretization by redistributing fluxes that correct the mass
distribution and the domain of dependence for arbitrary small cells and their neighbors.
Recent results provide stabilization techniques for nonlinear hyperbolic equations in one
spatial dimension \cite{may2022dod} and for linear systems in two dimensions
\cite{birke2023dod,birke2024error, streitbuerger2022}.
Proven stability properties of DoD-extended DG schemes include semidiscrete
stability for linear systems \cite{birke2023dod, streitbuerger2022} as well as fully-discrete
stability \cite{petri2025domain}. According to \cite{may2022dod}, the
first-order scheme is monotone.
The state redistribution method provides similar properties
\cite{TAYLOR2025,Berger2024}.

In the realm of hyperbolic balance laws, kinetic equations are a subclass describing the evolution of quantities (e.g., a density) depending on time and position space as well as velocity, where the model includes both linear transport of the considered quantity and a potentially nonlinear collision term.
When considered in diffusion scaling using a scaling parameter $\epsilon$, these models are subject to stiffness both in the transport and the collision term for small $\epsilon$ which requires some degree of implicitness for the time discretization of both contributors. Furthermore, in the limit $\epsilon\rightarrow 0$, the kinetic model in diffusion scaling generally relaxes towards a macroscopic model given by a parabolic limit equation which changes the properties of the underlying system entirely. One of the simplest model problems is
the telegraph equation, a linear two-velocity model that relaxes to the heat equation. Designing methods for the kinetic equation that
provide consistent approximations in the asymptotic limit is an active topic of current research.
DG schemes for the telegraph equation and for some extended systems are developed in
\cite{Jang_etal:2014,Jang_etal:2014_analysis,Peng2021schur}.

We combine the two topics of cut-cell meshes and kinetic models by extending the DoD stabilization
for DG methods on cut-cell meshes to the telegraph equation, ensuring stability
under feasible time step sizes, as well as preserving the asymptotic limit, whereby the limit scheme yields a
DG DoD discretization of the diffusion operator. To prove energy stability, we draw connections
to summation-by-parts (SBP) operators (e.g., \cite{MATTSSON2017diagonal,fernandez2014review,svard2014review}),
which are designed to mimic integration-by-parts and therefore
enable the transfer of energy properties of the continuous PDE to the semidiscretization.

The content of this article can be split up in two parts: The first part introduces
the DG method for the telegraph equation and provides general theoretical
results on energy stability by using periodic (upwind) SBP operators
(Sections~\ref{sec:preliminaries} and \ref{sec:analysis}) as well as asymptotic preservation of the fully discrete scheme.
The second part constructs the cut-cell
stabilization and implements it into the SBP framework
(Sections~\ref{sec:cut_cell_intro} and \ref{sec:SBP}). The full article is structured as follows:
In Section~\ref{sec:preliminaries}, we introduce the telegraph equation with its
spatial DG and temporal IMEX-RK discretizations. Furthermore, we tie the method to
periodic (upwind) SBP operators. In Section~\ref{sec:analysis}, we provide a
semidiscrete stability result under the assumption that periodic (upwind) SBP
operators are used to discretize the spatial derivatives. This result can be
considered as independent of further details of the semidiscretization. In
Section~\ref{sec:cut_cell_intro}, we review the domain-of-dependence stabilization
for cut-cell meshes for the linear advection equation, and apply it to the DG
semidiscretization of the telegraph equation introduced in
Section~\ref{sec:preliminaries}. Section~\ref{sec:SBP} verifies that
the DoD stabilization for central numerical fluxes results in a periodic SBP operator.
We further adapt the method for upwind numerical fluxes to also obtain periodic
upwind SBP operators in that case. Due to this result, the DoD scheme is included within the stability theory
of Section~\ref{sec:analysis}.
Finally, we verify our theoretical findings by numerical results in Section~\ref{sec:numerics}
and elaborate on the utility of specific stabilization choices.

\section{Preliminaries}
\label{sec:preliminaries}
\subsection{The telegraph equation}
We consider the telegraph equation in diffusion scaling,
a well-studied linear kinetic transport model
\cite{ortleb2024unconditional, Jang_etal:2014}.
In one spatial dimension, it reads
\begin{align}\label{eq:kin_model}
 \epsilon \partial_t f + v\partial_x f = \frac{1}{\epsilon}\left(\mean f-f\right)\,,
\end{align}
where $f(x,v,t)$ is the probability distribution function of particles depending on
the spatial position $x\in\Omega_x$, the velocity $v\in\Omega_v=\{-1,1\}$, and the
time $t\in\R^+_0$. Furthermore, the mean
\[\mean f = \frac12\left(f|_{v=-1}+f|_{v=1}\right)\]
can be considered as the macroscopic density of particles, depending only on $x$ and $t$.
Obviously, the average in velocity is normalized such that $\mean 1 = 1$ and $\mean v = 0$.

To construct asymptotic preserving schemes, the framework of micro-macro decomposition \cite{LemouMieussens:2008}
is widely used. Hereby, $f$ is orthogonally decomposed into $f=\rho +\epsilon g$, with
macroscopic density $\rho=\mean{f}$ and non-equilibrium part
$g=\frac1\epsilon(f-\rho) = \frac1\epsilon(f-\mean f)$ fulfilling $\mean g=0$.
The resulting micro-macro decomposed system reads
\begin{subequations}
\label{eq:micromacro}
\begin{align}
 \partial_t \rho + \partial_x \mean{vg} &= 0\,, \label{eq:micromacro1}\\
  \epsilon \partial_t g + \partial_x (vg) - \mean{\partial_x (vg)} + \frac1\epsilon v\partial_x\rho &= - \frac{1}{\epsilon}g\,, \label{eq:micromacro2}
\end{align}
\end{subequations}
and preserves the property $\mean g=0$ in time if it is fulfilled by the initial
condition $g_0=g(x,v,t=0)$.

Setting $g_i(x,t) = \frac1\epsilon\left(f(x,i,t)-\rho(x,t)\right),\ i=1,-1$, and
using $\tg:=g_1=-g_{-1}$ due to $\mean g=0$, the two-velocity model in micro-macro
decomposition can be rewritten as
\begin{subequations}
\label{eq:kinetic_rewritten}
\begin{align}
 \partial_t \rho + \partial_x \tg &= 0\,, \label{eq:kinetic_rewritten1}\\
  \epsilon^2 \partial_t \tg + \partial_x \rho  &= -\tg\,, \label{eq:kinetic_rewritten2}
\end{align}
\end{subequations}
and the third equation for $g_{-1}(x,t)$ becomes redundant. The system~\eqref{eq:kinetic_rewritten}
slightly differs from the classical reformulation of the scalar linear wave equation
as a system of first-order PDEs due to the source term in the second equation.

For $\epsilon\rightarrow 0$, this kinetic model converges to the linear heat
equation for the macroscopic density $\rho = \mean f$, given by
\begin{equation}\label{eq:heat_equation}
 \partial_t \rho = \partial_{xx}\rho\,.
\end{equation}
This follows formally by letting $\epsilon\rightarrow 0$ in~\eqref{eq:kinetic_rewritten2}
and inserting the resulting equation $\partial_x\rho=-\tg$ into~\eqref{eq:kinetic_rewritten1}.

\subsection{Spatial semidiscretization of the base scheme} \label{subsec:semidiscr_background}
As the base scheme, we consider the discontinuous Galerkin semidiscretization
as introduced by Jang et al.\ in \cite{Jang_etal:2014}.
To discretize the spatial domain $\Omega_x=[x_\text{min}, x_{\text{max}}]$,
we split it into $N$ cells
$E_i = (x_{i-\frac12}, x_{i+\frac12})$ for $i = 1, \hdots, N$ with vertices
$x_\text{min}=x_{\frac12} < x_{\frac32} < x_{\frac52} < \hdots < x_{N+\frac12} = x_\text{max}$.
Further, we consider the polynomial space of degree $p$, denoted by $\mathcal{P}^p$,
and the discrete function space
$$ \mathcal{V}_h^p(\Omega_x) = \left\{w_h \in L^2(\Omega) :w_h|_{E_i} \in \mathcal{P}^p, \; i \in 1,\hdots N \right\}.$$
Since the elements of $\mathcal{V}_h^p(\Omega)$ are not well-defined at the vertices,
we also introduce the jump and the mean value
 $$\llbracket w_h \rrbracket_{i+\frac12} = w_h(x_{i+\frac12})|_{E_{i}}-w_h(x_{i+\frac12})|_{E_{i+1}}, \quad
  \{w_h\}_{i+\frac12} = \frac{1}{2}\left(w_h(x_{i+\frac12})|_{E_{i}}+w_h(x_{i+\frac12})|_{E_{i+1}}\right).$$
To discretize~\eqref{eq:kinetic_rewritten}, we need to choose numerical fluxes
for the spatial derivatives.
In principle, this can be done by diagonalizing the linear hyperbolic system~\eqref{eq:kinetic_rewritten}
and applying standard upwind/downwind fluxes to the characteristic variables. However,
it is well-known that the resulting numerical scheme, even with implicit
time integration, leads to the wrong diffusion limit with a wrong (higher)
diffusivity coefficient in the asymptotic limit \cite{NaldiPareschi:1998}.
This can be prevented by employing different spatial discretizations for those
terms in the micro-macro decomposition containing $\partial_x$ before rewriting
\eqref{eq:micromacro} as the two-by-two system~\eqref{eq:kinetic_rewritten}.
There are two strategies considered by Jang et al.\ \cite{Jang_etal:2014}:
Using alternating upwind fluxes or central fluxes.
In general, we can write the semidiscrete method as a solution
$\rho_h(t), \tg_h(t)\in\mathcal{V}_h^p(\Omega_x)$,
such that
\begin{equation}\label{eq:semidiscretization_bilinearform}
    \begin{aligned}
        (\partial_t\rho_h(t), \phi_h)_{L^2(\Omega_x)}&+a_h^{\{-, +, \z\}}(\tg_h(t), \phi_h)=0, \; &\forall \phi_h \in\mathcal{V}_h^p(\Omega_x), \\
        (\partial_t\tg_h(t), \psi_h)_{L^2(\Omega_x)}&+\frac{1}{\epsilon^2}a_h^{\{+, -, \z\}}(\rho_h(t), \psi_h)+ \frac{1}{\epsilon}b_h(\tg_h(t), \psi_h) =
        -\frac{1}{\epsilon^2}(\tg_h(t), \psi_h)_{L^2(\Omega_x)},
        \; &\forall \psi_h \in\mathcal{V}_h^p(\Omega_x).
    \end{aligned}
\end{equation}
$a_h^\delta$ is in this case the standard DG discretization of the spatial derivative
with an upwind ($\delta=-$),
downwind ($\delta=+$) or central ($\delta = \z$) numerical flux, given by
\begin{equation}\label{eq:a_h}
a_h^\delta(u_h, w_h)= -\sum\limits_{i=1}^N\int\limits_{E_i}u_h\partial_x w_h \dif x
            + \sum\limits_{i=1}^N \mathcal{H}^\delta(u_i, u_{i+1})(x_{i+\frac12})\llbracket w_h \rrbracket_{i+\frac12},
\end{equation}
where
\begin{align*}
    u_i={u_h}|_{E_i}, \quad \mathcal{H}^{\delta}(a, b)(x)&:=
    \begin{cases}
        \hfil b(x), \quad &\delta = +,\\ \hfil a(x),
         \quad &\delta = -, \\ \dfrac{a(x)+b(x)}{2}, \quad &\delta = \z.
    \end{cases}
\end{align*}
The bilinear form $b_h(\tg_h, \psi)$ discretizes the term $\partial_x (vg) - \mean{\partial_x (vg)}$ of~\eqref{eq:micromacro2} for $v=1$,
that vanishes when transforming to the system~\eqref{eq:kinetic_rewritten}.
It is obtained by using the intermediate formulation
$$\partial_x (vg) - \mean{\partial_x (vg)}\stackrel{v=1}{=}\frac12\left(\partial_xg_1+\partial_x g_{-1}\right)$$
and discretizing the respective derivatives with upwind/downwind fluxes,
according to the inherent velocity directions of $g_1, g_{-1}$.
This leads to
$$b_h(\tg_h, \psi_h) = \frac{1}{2}\left(a_h^{-}(\tg_h, \psi_h)-a_h^{+}(\tg_h, \psi_h)\right).$$
According to \cite{Jang_etal:2014, Jang_etal:2014_analysis}, including this additional
discretization term results in a more favorable time step restriction of
$\Delta t = \mathcal{O}(\epsilon\min_i\{\Delta x_i\})$ in the convective regime
$\epsilon = \mathcal{O}(1)$ when specific implicit-explicit time stepping is applied, as considered in Section~\ref{subsec:imex_background_scheme}.

\subsection{Semidiscretization in matrix-vector/SBP formulation}\label{subsec:semidiscr_matrix_SBP}
In the following, we will write the semidiscretization~\eqref{eq:semidiscretization_bilinearform}
into a matrix-vector notation that aligns with the SBP framework.
This can be motivated by considering the
linear advection equation $\partial_t u+\vel\partial_x u=0$ with constant advection speed $\vel\in\R$ and its DG semidiscretization given by
\begin{equation*}
(\partial_t u_h(t),w_h) + \vel a_h^\delta(u_h(t), w_h) = 0,\quad \forall w_h \in\mathcal{V}_h^p(\Omega_x),
\end{equation*}
where $a_h^\delta$ is given in \eqref{eq:a_h}. The choice of $\delta$
is restricted by direction of the advection speed $\vel$: For $\vel>0$ one has to use a left-biased (upwind) or
central flux ($\delta \in \{-, \z\}$),
while for $\vel<0$ one has to use a right-biased or central flux ($\delta \in \{+, \z\}$) to obtain a stable scheme.
After choosing a basis for
$\mathcal{V}_h^p(\Omega_x)$ and a quadrature formula for the involved integrals,
we can write $a_h$ and the inner product $(\cdot,\cdot)_{L^2(\Omega_x)}$ using matrix-vector-multiplications, i.e.,
\begin{equation}\label{eq:bilinear_to_matrix_notation}
a_h^{\{+,-,\z\}}(u_h, w_h)=\underline{w}^TMD^{\{+,-,\z\}}\underline{u}, \quad (u_h, w_h)_{L^2(\Omega_x)} = \underline{w}^TM\underline{u},
\end{equation}
where the vectors $\underline{u},\underline{w}$ collect the nodal values of $u_h,w_h$ at the quadrature nodes, the matrices $D^{\{+,-,\z\}}$ are discrete derivative matrices, and $M$ is the mass matrix.

In general, given a bilinear operator $D$ and a suitable symmetric and positive definite matrix
$M$ (the mass matrix in all our applications), $D$ is an SBP operator, if it mimics
integration-by-parts by
\begin{equation}
MD + D^T M = \text{diag}\{-1, 0, \hdots, 0, 1\},
\end{equation}
where the matrix on the right-hand side incorporates the boundary conditions.
For simplicity, we stick to periodic boundary conditions, so usually $D$ is called a
periodic SBP operator, if
\begin{equation}\label{eq:_definition_of_sbp_operator}
    MD + D^T M = \mathbf 0.
\end{equation}
A dual-pair of operators $D^+, D^-$, employed with a positive symmetric matrix $M$,
are periodic upwind SBP operators, if
\begin{subequations}
\label{eq:_def_upwindSBP}
\begin{align}
    &MD^+ + {D^-}^TM = \mathbf{0},\label{eq:_def_upwindSBP1} \\
    &M\left(D^+ -D^-\right) \text{ is negative semidefinite}. \label{eq:_def_upwindSBP2}
\end{align}
\end{subequations}
SBP and upwind SBP properties of discontinuous Galerkin schemes have been investigated
in \cite{gassner2013skew,carpenter2014entropy,ranocha2021broadclass,ortleb2023stability}.
Therefore, it is well-known that the cell-wise formulation of the discontinuous
Galerkin scheme on each DG element results in a generalized SBP operator while
the interface terms containing numerical fluxes can be considered as
simultaneous approximation terms (SATs). Furthermore, including the numerical
fluxes on cell interfaces into the global DG operator on the complete spatial
domain yields a dual-pair of upwind SBP operators.

We want to take advantage of this SBP framework for the telegraph equation
and therefore proceed with
the semidiscretization \eqref{eq:semidiscretization_bilinearform}.
With \eqref{eq:bilinear_to_matrix_notation}, we consider
\begin{equation}\label{eq:D1_D2_pairs}
    \left(\Drho, \Dg\right) \in \left\{\left(D^\z, D^\z\right), \left(D^+, D^-\right), \left(D^-, D^+\right)\right\},
\end{equation}
that resemble the alternating upwind/downwind or central flux pairs and are used
for the spatial derivatives in the respective first and second component of the
PDE system.
The variational formulation \eqref{eq:semidiscretization_bilinearform} is then rewritten in matrix-vector-formulation by
\begin{equation}\label{eq:semidiscretization_matrixform}
    \begin{aligned}
        M\underline{\rho}_t&+M\Drho\underline{\tg}=\underline{0},\\
        M\underline{\tg}_t&+\frac{1}{\epsilon^2}M\Dg\underline{\rho}+\frac{1}{2\epsilon}M\left(D^- - D^+\right)\underline{\tg}=-\frac{1}{\epsilon^2}M\underline{\tg}.
    \end{aligned}
\end{equation}

\begin{remark}\label{rmk:SBP_bilinear_equivalence}
    System \eqref{eq:semidiscretization_matrixform} is solely a change in notation.
    Indeed, we can formulate the desired properties also in the shape of bilinear forms:
    The operator \eqref{eq:a_h} is a periodic SBP operator, if and only if it is skew symmetric,
    because the transposition of the matrix is equivalent to switching the arguments. Furthermore,
    the operators are independent of the choice of basis, as long as all integrals are
    evaluated exactly, which is the case for this linear equation.
    Analogously, the bilinear forms with upwind numerical fluxes yield
    periodic upwind SBP properties satisfying the equivalent conditions
    \begin{equation}\label{eq:bilinearform_upwind_sbp}
        \begin{aligned}
            a^{+}_h(u_h, w_h)& = -a^{-}_h(w_h, u_h), \\
            a^{+}_h(u_h, u_h)& - a^{-}_h(u_h, u_h) \le 0,
        \end{aligned}
    \end{equation}
    for all $u_h,w_h\in\mathcal{V}_h^p(\Omega_x)$. We will take advantage of both of the notations, depending on which form is more convenient
    for the analysis.
\end{remark}

\subsection{Time discretization of the base scheme}\label{subsec:imex_background_scheme}
As the system \eqref{eq:semidiscretization_matrixform} becomes stiff for small $\epsilon > 0$,
an implicit treatment of the respective terms is necessary to avoid a severe
time step restriction of order $\mathcal{O}(\epsilon^2\Delta x)$. Fully
implicit time integration comes with the additional cost of solving a linear
system in every internal stage. To take computational advantage of both  classes,
implicit-explicit (IMEX) approaches are usually considered. Further, a careful choice
of the IMEX splitting is necessary to finally obtain an asymptotic preserving scheme.
In this work, we align with the strategy of \cite{Jang_etal:2014} and discretize
only the terms of order ${\mathcal O}(1/\epsilon^2)$ in equation \eqref{eq:micromacro2}
implicitly.

An IMEX additive Runge-Kutta method with $s$ stages applied to the system of ODEs
$u_t = f(u)+g(u), u(0)=u^0,$
is given by
\begin{equation}\label{eq:IMEX_scheme_arbitrary}
    \begin{aligned}
    u^{(i)} &= u^n + \Delta t \sum_{j=1}^{i-1} \tilde{a}_{ij} f(u^{(j)})
            + \Delta t \sum_{j=1}^{i} a_{ij} g(u^{(j)}), \quad i = 1, \ldots, s, \\
    u^{n+1} &= u^n + \Delta t \sum_{i=1}^{s} \tilde{b}_i f(u^{(i)})
            + \Delta t \sum_{i=1}^{s} b_i g(u^{(i)}),
    \end{aligned}
\end{equation}
where $f(u)$ denotes the non-stiff operator that shall be treated
explicitly, and $g(u)$ is the stiff part that is treated implicitly.
A specific IMEX-RK method is defined by the Butcher matrices $\tilde{A}=(\tilde{a}_{ij}), \; A=(a_{ij})$
and the weights $\tilde{b}=(\tilde{b}_1, \hdots, \tilde{b}_s),\; b = (b_1, \hdots, b_s)$.

For the semidiscrete system~\eqref{eq:semidiscretization_matrixform},
we treat the stiff terms scaling with $1/\epsilon^2$
implicitly and the remaining terms explicitly, such that
\begin{equation}\label{eq:imex_splitting}
    f\left((\underline{\rho}, \underline{\tg})^T\right)= \begin{pmatrix}\mathbf{0}&-\Drho \\\mathbf{0} & \frac{1}{2\epsilon}\left(D^+-D^-\right) \end{pmatrix}
    \begin{pmatrix}\underline{\rho}\\\underline{\tg}\end{pmatrix},
    \quad g\left((\underline{\rho}, \underline{\tg})^T\right)= \begin{pmatrix}\mathbf{0}& \mathbf{0} \\-\frac{1}{\epsilon^2}\Dg & -\frac{1}{\epsilon^2}\mathbf{\text{I}} \end{pmatrix}
    \begin{pmatrix}\underline{\rho}\\\underline{\tg}\end{pmatrix}.
\end{equation}
The resulting IMEX-RK-DG scheme can be solved explicitly by first
computing $\rho{(i)}$, which results in a linear equation for every
component of $\tg^{(i)}$ in every internal stage.

We recall two widely considered classes of IMEX-RK methods (see, e.g., \cite[Section~3.2]{boscarino2024}):
\begin{definition}\label{def:IMEX_RK_types}
    An IMEX-RK method is of type
    \begin{itemize}
        \item I, if $a_{ii}\neq 0, i=1\hdots s$,
        \item II, if $a_{11}=0, \; a_{ii}\neq 0, i=2, \hdots, s$,
        \item ARS, if it is of type II and additionally $b_1=a_{i1}=0, i=2, \hdots, s$.
    \end{itemize}
    An IMEX-RK method is globally stiffly accurate (GSA), if
    $$a_{si}=b_i, \quad \tilde{a}_{si}=\tilde{b}_i, \;i=1,\hdots,s,$$
    i.e., if the implicit part is stiffly accurate and the explicit part has the
    first-same-as-last (FSAL) property.
\end{definition}
GSA IMEX-RK methods have the useful property, that the value of the last internal stage equals
the final solution update, i.e., $u^{n+1} = u^{(s)}$. Also, if the implicit part of the method
is in addition \textit{A}-stable, the scheme is \textit{L}-stable, a useful property for
treating stiff equations (see \cite{boscarino2024},\cite[Chapter IV.3]{hairer1996solving} for more information).
Based on the upcoming analysis and simulations, we restrict ourselves in the
following to type I and GSA type II schemes. These schemes inherit a time step restriction
$\Delta t = {\mathcal O}(\epsilon \Delta x + \Delta x^2)$ and are AP explicit
$\epsilon\rightarrow0$
(see \cite{Jang_etal:2014,Jang_etal:2014_analysis} and
Theorem~\ref{theorem:ap_property} below).

Restricting to this subclass of IMEX schemes is required to obtain
provably asymptotic preserving schemes. Further, we also noticed, that using for example
Type II methods that are not GSA can lead to unstable results for
$\epsilon \rightarrow 0$ under the envisioned
time step restriction $\Delta t = C_1\epsilon \Delta x + C_2 \Delta x^2$.

\subsection{The asymptotic limit}
If we let
$\epsilon \rightarrow 0$ in \eqref{eq:semidiscretization_matrixform}, we obtain
the consistent DG semidiscretization
\begin{equation}\label{eq:heat_semidiscretization}
        M\underline{\rho}_t = M \Drho\Dg\underline{\rho}
\end{equation}
of the heat equation \eqref{eq:heat_equation}.
If $\Drho$ and $\Dg$ incorporate alternating upwind fluxes, this is the minimal
dissipation local DG (LDG) scheme \cite{cockburn2007minimal}. By using central fluxes,
this is the method of Bassi and Rebay (BR1) \cite{BassiRebay:97JCP}.

In \cite[Prop. 3.4]{Jang_etal:2014}, the question of asymptotic convergence using a GSA
IMEX scheme for \eqref{eq:imex_splitting} gets addressed. The limiting scheme is AP
explicit with time step restriction
$\Delta t = {\mathcal O}(\Delta x^2)$ for $\epsilon\rightarrow0$, i.e., it is given by
the explicit part of the chosen IMEX method \eqref{eq:IMEX_scheme_arbitrary} applied to
\eqref{eq:heat_semidiscretization}. We will extend this result while taking the
upcoming cut-cell stabilization into account.

\section{Semidiscrete stability and asymptotic analysis}
\label{sec:analysis}
In this section, we provide two theoretical results based on the SBP framework,
i.e., by using the operators obtained in \eqref{eq:bilinear_to_matrix_notation}.
The following results are not
restricted to the introduced DG methods and can be applied to every semidiscretization
of the type \eqref{eq:semidiscretization_matrixform}, while the semidiscrete
stability result additionally requires the involved operators to fulfill either the
periodic SBP or the periodic upwind SBP conditions. Indeed, the aim of the
upcoming Sections~\ref{sec:cut_cell_intro} and~\ref{sec:SBP} is to introduce
analogue operators $D^\z, \; D^+,\; D^-$, that inherit the DoD-DG cut-cell method
and will be constructed and proven to fulfill periodic (upwind) SBP properties.
The following results are formulated, such that they cover the DoD-DG
cases and in general a semidiscretization using periodic (upwind) SBP operators,
as the resulting properties are not new for the background DG schemes.

\subsection{Semidiscrete stability}
As introduced previously, we assume we have a pair $\bigl(\Drho, \Dg\bigr)$ of the type \eqref{eq:D1_D2_pairs}, that
discretizes the respective spatial derivatives of $\rho$ and $\tg$ in the telegraph
equation to obtain a semidiscretization of the form~\eqref{eq:semidiscretization_matrixform} with periodic (upwind) SBP properties, i.e., with
\begin{equation}\label{eq:general_skew_symmetry}
    M\Drho=-(\Dg)^T M,
\end{equation}
corresponding to \eqref{eq:_definition_of_sbp_operator} for central SBP operators
and \eqref{eq:_def_upwindSBP1} for their upwind SBP version and with the additional dissipative property \eqref{eq:_def_upwindSBP2} for the chosen basic SBP operators $D^-$ and $D^+$ in \eqref{eq:D1_D2_pairs} and \eqref{eq:semidiscretization_matrixform}.
This leads to the following result:
\begin{theorem}\label{theorem:semidiscr_stability}
Consider the semidiscretization \eqref{eq:semidiscretization_matrixform}
of the telegraph equation using a pair of periodic \mbox{(upwind) SBP} operators
$\bigl(\Drho, \Dg\bigr)$ of the type \eqref{eq:D1_D2_pairs}.
Then, this semidiscrete system is stable with respect to
the weighted inner-product $\left(u,v\right)_{\underline{M}}$, i.e., the semidiscrete
solution $\left(\rho, \tg\right)^T$ satisfies
\begin{equation*}
   \partial_t \left\|\rhoj\right\|_{\underline{M}} \le 0,
\end{equation*}
where $\underline{M} = \operatorname{diag}(M, \epsilon^2 M)$.
\end{theorem}
\begin{proof}
    We consider the semidiscretization in the matrix-vector form and perform a
    straightforward energy estimate
    \begin{align*}
    \frac{1}{2}\partial_t \left\|\rhoj\right\|^2_{\underline{M}}&= \rhoj^T\underline{M}\partial_t\rhoj \\
    &= \rhoj^T
    \begin{pmatrix} \mathbf{0}&-M\Drho \\
    -M\Dg&\mathbf{0} \\ \end{pmatrix}
    \rhoj + \frac{\epsilon}{2}\rhoj^T
    \begin{pmatrix} \mathbf{0}&\mathbf{0} \\
     \mathbf{0}& M(D^+-D^-)\\ \end{pmatrix}\rhoj -\rhoj^T
    \begin{pmatrix}0 & 0\\0 & M \end{pmatrix}\rhoj \\
    &=\rhoj^T
    \begin{pmatrix} \mathbf{0}&-M\Drho \\
    -M\Dg&\mathbf{0} \\ \end{pmatrix}
    \rhoj + \frac{\epsilon}{2} \tg^TM\left(D^+-D^-\right)\tg -\tg^TM\tg \\
    &\le
    \rhoj^T\begin{pmatrix} \mathbf{0}&-M\Drho \\
    -M\Dg&\mathbf{0} \\ \end{pmatrix}\rhoj
    \end{align*}
    by using the property \eqref{eq:_def_upwindSBP2} of the periodic upwind SBP operators for the
    last inequality. Using
    \eqref{eq:general_skew_symmetry}, the right-hand side term can then be rewritten as
    \begin{equation}\label{eq:skew_symmetry_fluxsystem}
            \rhoj^T\begin{pmatrix} \mathbf{0}&-M\Drho \\
    -M\Dg&\mathbf{0} \\ \end{pmatrix} \rhoj
    =
       \rhoj^T \begin{pmatrix} \mathbf{0}&{\Dg}^TM \\
    -M\Dg&\mathbf{0} \\ \end{pmatrix}\rhoj = 0,
    \end{equation}
    because of the skew-symmetry of the involved block matrix.
\end{proof}
\subsection{Formal asymptotic limit}
In this section, we perform a formal asymptotic analysis for semidiscretizations
of the form \eqref{eq:semidiscretization_matrixform}, to which we apply the
aforementioned IMEX schemes.  We want to emphasize in advance, that the used
techniques are standard, and the main takeaway is that by applying
the upcoming DoD stabilization,
the resulting change of the semidiscretization
does not impact the questions of asymptotic preservation at all, as long as
the time discretization is chosen carefully. For this purpose, we consider again a pair of upwind SBP operators~\eqref{eq:D1_D2_pairs}
to obtain the semidiscretization \eqref{eq:semidiscretization_matrixform}. In the limit $\epsilon\rightarrow 0$, we should obtain a consistent
full discretization of the limit equation (the linear heat equation), i.e., a reasonable time discretization of the semidiscrete scheme \eqref{eq:heat_semidiscretization}.

Since for asymptotic preservation, the crucial part of the scheme is the choice of IMEX method and splitting, we select the same option as proposed in \cite{Jang_etal:2014}
for the background DG method and extend this to an arbitrary pair $\bigl(\Drho, \Dg\bigr)$ with (upwind) SBP property.
\begin{theorem}\label{theorem:ap_property}
Consider an IMEX scheme
\begin{itemize}
    \item of type I
    \item or of type II with GSA property and well-prepared initial data,
          i.e., $\tg^0 = -\Dg \rho^0$,
\end{itemize}
applied to the splitting \eqref{eq:imex_splitting}. Then, in the limit $\epsilon\rightarrow 0$, this scheme reduces to the
explicit part of the IMEX scheme applied to~\eqref{eq:heat_semidiscretization}.
\end{theorem}
\begin{proof}
The internal stages $k=1, \hdots, s$ of the IMEX-RK scheme applied to the telegraph equation take the form
\begin{equation}\label{eq:imex_telegraph_scheme}
    \begin{aligned}
    \rho^{(k)}&=\rho^n-\Delta t \sum\limits_{i=1}^{k-1}\widetilde{a}_{ki}\Drho \tg^{(i)},\\
    \tg^{(k)}&=\tg^n+\frac{\Delta t}{2\epsilon}\sum\limits_{i=1}^{k-1}\widetilde{a}_{ki}\left(D^+-D^-\right)\tg^{(i)}
                -\frac{\Delta t}{\epsilon^2}\sum\limits_{i=1}^{k}a_{ki}\left(\Dg\rho^{(i)}+\tg^{(i)}\right).
    \end{aligned}
\end{equation}
Taking the limit $\epsilon \rightarrow 0$ solely influences the second equation and results in the restriction
\begin{equation}\label{eq:APrestriction}
    \sum\limits_{i=1}^{k}a_{ki}\left(\Dg\rho^{(i)}+\tg^{(i)}\right)=0, \quad k=1, \hdots, s.
\end{equation}
For now, we assume that
\begin{equation}\label{eq:first_internal_stage_limit}
    \tg^{(1)}=-\Dg\rho^{(1)}=-\Dg\rho^n
\end{equation}
holds and prove it later.
Then, since $a_{kk}\not=0$ for $k>1$, iteration of \eqref{eq:APrestriction} yields
$$\tg^{(k)}=-\Dg\rho^{(k)} \quad k = 1\hdots, s.$$
Now, we can insert the internal stages of $\tg$ into the internal stages and final update of $\rho$ to obtain
\begin{align*}
\rho^{(k)}&=\rho^n+\Delta t \sum\limits_{i=1}^{k-1}\widetilde{a}_{ki}\Drho\Dg\rho^{(i)},\\
\rho^{n+1}&=\rho^n+\Delta t \sum\limits_{i=1}^{s}\widetilde{b}_{i}\Drho\Dg\rho^{(i)},
\end{align*}
which is exactly the explicit treatment of the system of ODEs~\eqref{eq:heat_semidiscretization}.

It remains to show, that \eqref{eq:first_internal_stage_limit} holds true. For type I schemes, this is
immediately fulfilled because with $a_{11}\not=0$, equation \eqref{eq:APrestriction} directly yields
$\tg^{(1)}=-\Dg\rho^{(1)}=-\Dg\rho^n$.
As for type II schemes $a_{11}=0$, we have $\tg^{(1)}=\tg^n$ and need to consider the well-prepared initial
data for the second component, given by $\tg^{0}= -\Dg\rho^{0}$. This is the equivalence
of~\eqref{eq:first_internal_stage_limit} for the first time step. Because of the GSA property
(see Definition~\ref{def:IMEX_RK_types}), we know that $\tg^{n+1}=\tg^{(s)}=-\Dg\rho^{(s)}=-\Dg\rho^{n+1}$
holds for every $n$ and therefore~\eqref{eq:first_internal_stage_limit} is true for every $n$.
\end{proof}

\section{The cut-cell problem and the DoD stabilization}
\label{sec:cut_cell_intro}
Next, we introduce the cut-cell problem in one spatial dimension and
present the DoD stabilization for linear transport as a starting
point to extend it to the linear kinetic model.
\subsection{The cut-cell problem}\label{subsec:cut_cell_problem}
We include a cut at an arbitrary position inside a uniform partition of $\Omega_x$, such
that we can write the vertices as
$x_\text{min}=x_{\frac12} < x_{\frac32} < \hdots < x_{c-\frac12} < x_{c+\frac12} < \hdots < x_{N-\frac12} < x_{N+\frac12} = x_\text{max}$
with the cell sizes
\begin{align*}
    \abs{E_i} & = \Delta x, \quad i \in \{1,\hdots, N\} \setminus \{c-1, c, c+1\}, \quad
    \abs{E_c} = \alpha \Delta x.
\end{align*}
Hereby, $\alpha\le1/2$ to emphasize that $E_c$ is the small cut cell.
Depending on the position of the cut, $E_{c-1}$ is of size $(1-\alpha)\Delta x$
and $E_{c+1}$ is of size $\Delta x$ or vice-versa. The latter case is displayed in Figure~~\ref{fig:cut_cell_setting}.
\begin{figure}[htbp]
    \centering
   \begin{tikzpicture}
        \draw[gray, thick] (0,0) -- (12.5,0);
        \draw[gray, thick, dotted]  (0,0) -- (-0.5, 0);
        \filldraw[black] (1.25,-0.25)  node[anchor=north]{$\Delta x$};
        \filldraw[black] (1.25,0.0)  node[anchor=south]{$E_{(c-2)}$};

        \draw[gray, thick]  (2.5,1/4) -- (2.5, -1/4);
        \draw[gray, thick, dotted]  (2.5,0) -- (2.5, -4/5);
        \filldraw[black] (2.5,-4/5)  node[anchor=north]{$x_{c-\frac32}$};
        \filldraw[black] (3.75,-0.25)  node[anchor=north]{$\Delta x$};
        \filldraw[black] (3.75,0.0)  node[anchor=south]{$E_{(c-1)}$};

        \draw[gray, thick]  (5,1/4) -- (5, -1/4);
        \draw[gray, thick, dotted]  (5,0) -- (5, -4/5);
        \filldraw[black] (5,-4/5)  node[anchor=north]{$x_{c-\frac12}$};
        \filldraw[black ] (5.5,-0.25)  node[anchor=north]{$\alpha \Delta x$};
        \filldraw[black] (5.5,0.0)  node[anchor=south]{$E_{(c)}$};

        \draw[gray, thick]  (6,1/4) -- (6, -1/4);
        \draw[gray, thick, dotted]  (6,0) -- (6, -4/5);
        \filldraw[black] (6,-4/5)  node[anchor=north]{$x_{c+\frac12}$};
        \filldraw[black] (6.75,-0.25)  node[anchor=north]{$(1-\alpha)\Delta x$};
        \filldraw[black] (6.65,0.0)  node[anchor=south]{$E_{(c+1)}$};

        \draw[gray, thick]  (7.5,1/4) -- (7.5, -1/4);
        \draw[gray, thick, dotted]  (7.5,0) -- (7.5, -4/5);
        \filldraw[black] (7.5,-4/5)  node[anchor=north]{$x_{c+\frac32}$};
        \filldraw[black] (8.75,-0.25)  node[anchor=north]{$\Delta x$};
        \filldraw[black] (8.75,0.00)  node[anchor=south]{$E_{(c+2)}$};

        \draw[gray, thick]  (10,1/4) -- (10, -1/4);
        \draw[gray, thick, dotted]  (10,0) -- (10, -4/5);
        \filldraw[black] (10,-4/5)  node[anchor=north]{$x_{c+\frac52}$};
        \filldraw[black] (11.25,-0.25)  node[anchor=north]{$\Delta x$};
        \filldraw[black] (11.25,0.00)  node[anchor=south]{$E_{(c+3)}$};
        \draw[gray, thick, dotted]  (12.5,0) -- (13, 0);
        \end{tikzpicture}
        \caption{One-dimensional cut-cell setting.}
        \label{fig:cut_cell_setting}
\end{figure}
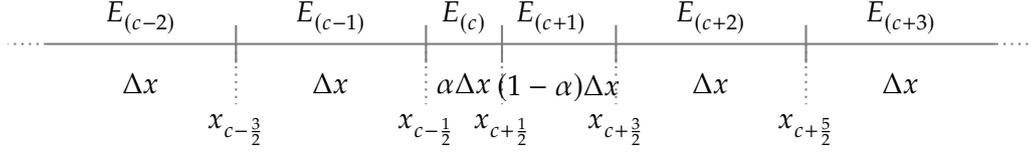
This larger cut cell is not in need of stabilization, as its size is bounded by $\Delta x/2$.

An arbitrarily small cell imposes several challenges for transport terms. In one spatial dimension, this reduces
to a mass-overload inside the small cell and an incorrect domain of dependence of the outflow cell, if the CFL condition
is chosen based on the background mesh size $\Delta x$. We want to avoid a standard time step restriction
dictated by the small cell $\alpha \Delta x$, as this leads to unfeasible small time steps
$\Delta t \le C \alpha\Delta x$. Therefore, we require a stabilization mechanism.

\subsection{The domain-of-dependence stabilization}
To stabilize the small cut cells, we add a penalty term at the semidiscrete level,
that redistributes the masses between the small cut cell and its neighbors. Further,
this resolves the disrupted domain of dependence of the neighboring cells.
Exemplarily, we display the current version applied for the linear transport
equation, following \cite{may2022dod}. Consider
\begin{equation}\label{eq:linear_tansport}
    \partial_tu + \vel\partial_x u = 0,
\end{equation}
which we semidiscretize by the (upwind) SBP operators previously considered for the transport terms in the telegraph equation: Find
$u_h(t) \in \mathcal{V}_h^p(\Omega_x)$, such that
\begin{equation*}
(\partial_t u_h(t), w_h)_{L^2(\Omega_x)}+\vel a_h^{\{+, -, \z\}}(u_h(t), w_h)=0\quad \forall w_h \in \mathcal{V}_h^p(\Omega_x),
\end{equation*}
where $a_h^{\{+, -, \z\}}$ is again defined by~\eqref{eq:a_h}. Before we introduce the
stabilization term, we need an operator that extends the polynomial from a cell
onto neighboring cells.
\begin{definition}[Extension operator]
        Recall the notation of Section~\ref{subsec:semidiscr_background} and
        consider $u_h\in \mathcal{V}_h^p(\Omega_x)$ and a cell $E_i\subset\Omega_x$.
        We extend the domain of $u_i$ from $E_i$ to $\Omega_x$ by
        $$u_i\in \mathcal{P}^p(\Omega_x)\; \text{with} \; u_i|_{E_i} = u_h|_{E_i}.$$
        Note that this is just the evaluation of the uniquely defined polynomial $u_h|_{E_i}$ outside of $E_i$.
\end{definition}

In the following, we will introduce the DoD stabilization by the two operators
$J_h^{0, c, \delta}, J_h^{1, c, \delta}$, for $\delta\in\{+,-,z\}$, to treat the volume and flux terms of \eqref{eq:a_h} separately.
To correct the mass-flow into cell
$E_c$, we introduce the stabilization term
\begin{equation}\label{eq:def_J0_bilinar}
\begin{aligned}
    J^{0, c, \delta}_h(u_h, w) &= \eta_c\left(\mathcal{H}^{\delta}(u_{c-1}, u_{c+1})(x_{c-\frac12})-\mathcal{H}^{\delta}(u_{c-1}, u_c)(x_{c-\frac12})\right)\llbracket w\rrbracket_{c-\frac12} \\
                    &+ \;\eta_c\left(\mathcal{H}^{\delta}(u_{c-1}, u_{c+1})(x_{c+\frac12})-\mathcal{H}^{\delta}(u_{c}, u_{c+1})(x_{c+\frac12})\right)\llbracket w\rrbracket_{c+\frac12},
\end{aligned}
\end{equation}
which replaces a fraction $\eta_c \in [0, 1]$ of the already existing flux by a new flux,
possibly containing an extended solution from another cell.
The following second stabilization term accounts for the mass distribution inside
the volumes, both of the cut cell and its neighbors. It is given by
\begin{equation}\label{eq:def_J1_bilinar}
    \begin{aligned}
    J_{h}^{1,c, {\delta}}(u^h, w)
    &= \eta_{c} \sum_{j \in \mathcal{I}_c} K(j)
    \int_{E_c} \left( \mathcal{H}^{\delta}(u_{c-1}, u_{c+1})
    - f(u_j) \right)
    \cdot \partial_x w_j \, \mathrm{d}x \\
    &\quad + \eta_{c} \sum_{j \in \mathcal{I}_c} K(j)
    \int_{E_c} \left( \mathcal{H}^{\delta}_a(u_{c-1}, u_{c+1})\, u_j \right)
    \cdot \partial_x w_{c-1} \, \mathrm{d}x \\
    &\quad + \eta_{c} \sum_{j \in \mathcal{I}_c} K(j)
    \int_{E_c} \left( \mathcal{H}^{\delta}_b(u_{c-1}, u_{c+1})\, u_j \right)
    \cdot \partial_x w_{c+1} \, \mathrm{d}x,
    \end{aligned}
\end{equation}
where $\mathcal{I}_c=\{c-1, c, c+1\}$ and $\mathcal{H}^{\delta}_a(a, b),\ \mathcal{H}^{\delta}_b(a, b)$ are the
partial derivatives of the respective first and second component of $\mathcal{H}^\delta$.
Further, we align with the notation of \cite{may2022dod} and write
$K(c-1)=L_c$, $K(c)=-1,\ K(c+1)=R_c$ with $L_c+R_c = 1$, with the intention to specify $L_c$ and $R_c$
according to the flow directions.

The full stabilized semidiscrete system for \eqref{eq:linear_tansport} is then given by
\begin{equation}\label{eq:dod_semidiscretization}
    \left( \partial_t u_h(t), w_h\right)_{L^2(\Omega)}
    + \vel a_h^\delta\left( u_h(t), w_h \right)
    + \vel J^{\delta}_h\left( u_h(t), w_h \right) = 0, \quad \forall w_h \in \mathcal{V}_h^p,
\end{equation}
where $J^{\delta}_h\left( u_h(t), w_h \right) = J^{0, c, {\delta}}_h\left( u_h(t), w_h \right) + J^{1, c, {\delta}}_h\left( u_h(t), w_h \right)$.

\begin{remark}
        For simplicity, stabilization was introduced for a single cut cell.
        Multiple cut cells can be treated analogously by adding stabilization
        for the respective cells.
        Let $\mathcal{S}=\{E_j : |E_j|\le \Delta x/2\}$ be the set of all small cells
        requiring stabilization. Then $J_h^{\delta}$ in \eqref{eq:dod_semidiscretization}
        is given by
    $$ J_h^{0,\delta}\left( u_h, w_h \right) = \sum\limits_{j \in \mathcal{S}}J_h^{0,j, {\delta}}\left( u_h, w_h \right),
    \quad  J_h^{1, {\delta}}\left( u_h, w_h \right) = \sum\limits_{j \in \mathcal{S}}J_h^{1,j, {\delta}}\left( u_h, w_h \right).$$
\end{remark}

\subsection{DoD stabilization for the telegraph equation}
To stabilize the semidiscrete system \eqref{eq:semidiscretization_bilinearform} in case of small cut cells,
our first approach is to stabilize all transport terms in an appropriate manner according to the respective
numerical fluxes used within the background DG scheme by the previously described DoD stabilization mechanism. This results in
\begin{equation}\label{eq:dod_semidiscr_bilinearform}
    \begin{aligned}
        (\partial_t\rho_h(t), \phi_h)_{L^2(\Omega_x)}&+a_h^{\{-, +, \z\}}(\tg_h(t), \phi_h)+J^{\{-, +, \z\}}_h\left(\tg_h(t), \phi_h \right)=0,
        \quad &\forall \phi_h \in\mathcal{V}_h^p(\Omega_x), \\
        (\partial_t\tg_h(t), \psi_h)_{L^2(\Omega_x)}&+\frac{1}{\epsilon^2}\left(a_h^{\{+, -, \z\}}(\rho_h(t), \psi_h)+J^{\{+, -, \z\}}_h\left(\rho_h(t), \psi_h \right)\right)\\
        &+ \frac{1}{\epsilon}\left(b_h(\tg_h(t), \psi_h) +\frac{1}{2}\left(J^{-}_h\left(\tg_h(t), \psi_h \right)-J^{+}_h\left(\tg_h(t), \psi_h \right)\right)\right)\\
        &= -\frac{1}{\epsilon^2}(\tg_h(t), \psi_h)_{L^2(\Omega_x)}, \quad &\forall \psi_h \in\mathcal{V}_h^p(\Omega_x).
    \end{aligned}
\end{equation}
As described in Section~\ref{subsec:semidiscr_matrix_SBP}, we can also write
this stabilized semidiscretization in the matrix-vector form \eqref{eq:D1_D2_pairs}, \eqref{eq:semidiscretization_matrixform}
by the adapted choice
\begin{equation}\label{eq:bilinear_to_matrix_notation_dod}
a_h^{\{+,-,\z\}}(u_h, w_h) + J_h^{\{+,-,\z\}}(u_h, w_h)=\underline{w}^TMD^{\{+,-,\z\}}\underline{u}.
\end{equation}
We want to emphasize that the above preliminary DoD scheme~\eqref{eq:dod_semidiscr_bilinearform} for the
telegraph equation is
meant as a first, natural extension based on the transport term. In the
following Section~\ref{sec:SBP}, we will analyze the method further
and recognize that we require a slight adjustment in order to obtain periodic
(upwind) SBP properties and therefore semidiscrete stability.
In particular, this means that the method~\eqref{eq:dod_semidiscr_bilinearform} is not the final DoD-stabilization in the case of
alternating numerical upwind fluxes.
However, it is sufficient and in its final form in the case of central numerical fluxes because in that case,
the DoD modification directly inherits the classical SBP property of the background DG scheme.

\begin{remark}
We did not specify the parameters $\eta_c, \; L_c,\; R_c$ yet. While the choice
of $\eta_c\in [0,1]$ seems to be arbitrary concerning semidiscrete properties,
$L_c$ and $R_c$ will matter and therefore be discussed in the next section.
\end{remark}

\begin{remark}
By choosing the IMEX splitting as in \eqref{eq:imex_splitting},
one may wonder whether it is necessary to stabilize the transport
term in the second equation of \eqref{eq:dod_semidiscr_bilinearform},
(i.e., inherited from $\Dg$), as this term is treated implicitly anyway.
But by looking closer at the fully discrete scheme, one recognizes that although
the operator $\Dg$ acts on the updated value of $\tg$ (of the full time step or
inside an internal stage), the information inside $\Dg$ still gets propagated
explicitly by the method, which leads to instabilities by choosing $\Delta t$
independent of the small cells. This also highlights the major difference to
the stiffness that originates from small $\epsilon^2$, which gets covered by
the implicit treatment of the source term $-\tg/\epsilon^2$ in the second equation.
\end{remark}

\section{SBP properties of the DoD stabilization}
\label{sec:SBP}
In the following, the goal is to retain the SBP properties when
the DoD-stabilization is applied to the DG method with an underlying cut-cell mesh
as introduced in Section~\ref{subsec:cut_cell_problem}. Without loss of generality,
this result extends to multiple small cells that need stabilization.
In particular, we consider the DoD stabilized scheme \eqref{eq:dod_semidiscretization}
for the linear advection equation either in variational formulation using bilinear
forms or in matrix-vector notation \eqref{eq:bilinear_to_matrix_notation_dod}.
Similar to subsection~\ref{subsec:semidiscr_matrix_SBP}, we obtain operators,
that will be
analyzed and adapted with respect to periodic (upwind) SBP properties.
We first provide an example using upwind fluxes for the first order scheme in
detail and show its upwind-SBP property in terms of the matrix-vector notation.
Then, we continue with
the case of central numerical fluxes and higher order upwind variants
in the more compact notation of bilinear forms.

\subsection{SBP properties of DoD with upwind numerical fluxes for the finite volume case}\label{subsec:SBP_dod_p=0}
In the following, we show that for the finite volume ($p=0$) case, the DoD
stabilized semidiscrete scheme given by \eqref{eq:def_J0_bilinar} retains the
periodic SBP properties of the unstabilized formulation.
More precisely, on the full grid, we obtain a dual-pair of periodic upwind SBP
operators $D^-,D^+$.

For this purpose, the DoD stabilized scheme is applied to the linear advection
equation with advection speed $\vel=\pm1$ on a periodic grid.

For $\vel=1$, i.e. $u_t+u_x=0$, using upwind numerical fluxes $\mathcal{H}^-$
in \eqref{eq:def_J0_bilinar} results in the semidiscrete scheme
\begin{align*}
 \frac{d}{dt}u_{c-1} + \frac{1}{\Delta x}\left(u_{c-1}-u_{c-2}\right) &= 0\\
  \frac{d}{dt}u_{c} + \frac{1-\eta_c}{\alpha \Delta x}\left(u_{c}-u_{c-1}\right) &= 0\\
  \frac{d}{dt}u_{c+1} + \frac{1}{(1-\alpha) \Delta x}\left(u_{c+1}-u_{c}\right)-\frac{\eta_c}{(1-\alpha) \Delta x}\left(u_{c-1}-u_{c}\right) &= 0\\
  \frac{d}{dt}u_{c+2} + \frac{1}{\Delta x} \left(u_{c+2}-u_{c+1}\right) &= 0.
 \end{align*}
In matrix-vector notation, this yields $\frac{d}{dt}\underline u + D^-\underline u = \underline 0$ with
\[  D^-=\frac{1}{\Delta x}\begin{pmatrix}
\ddots & & & & \\-1 & 1 & & &\\ & \frac{\eta_c-1}{\alpha} & \frac{1-\eta_c}{\alpha} & &\\& -\frac{\eta_c}{1-\alpha} & \frac{\eta_c-1}{1-\alpha} & \frac{1}{1-\alpha} &\\
& & & -1 & 1 &\\ & & & & & \ddots
\end{pmatrix}\]

On the other hand, for $\vel=-1$, i.e. $u_t-u_x=0$, downwinding using numerical fluxes
$\mathcal{H}^+$ results in
\begin{align*}
 \frac{d}{dt}u_{c-2} - \frac{1}{\Delta x}\left(u_{c-1}-u_{c-2}\right) &= 0\\
  \frac{d}{dt}u_{c-1} -\frac{1}{\Delta x} \left(u_{c}-u_{c-1}\right) - \frac{\eta_c}{\Delta x}\left(u_{c+1}-u_{c}\right)&= 0\\
\frac{d}{dt}u_{c} -\frac{1-\eta_c}{\alpha \Delta x}\left(u_{c+1}-u_{c}\right)&= 0\\
  \frac{d}{dt}u_{c+1} - \frac{1}{(1-\alpha) \Delta x}\left(u_{c+2}-u_{c+1}\right) &= 0
 \end{align*}
yielding $\frac{d}{dt}\underline u - D^+\underline u = \underline 0$ with
\[  D^+=\frac{1}{\Delta x}\begin{pmatrix}
\ddots &  & & & &\\ & -1 & 1 & & & \\ & & -1 & 1-\eta_c & \eta_c &\\ & & & \frac{\eta_c-1}{\alpha} & \frac{1-\eta_c}{\alpha} &\\[.5mm]
& & & & -\frac{1}{1-\alpha} & \frac{1}{1-\alpha} &\\ & & & & & -1 & 1\\ & & & & & & \ddots
\end{pmatrix}.\]
\begin{lemma}
Choosing the norm matrix $M=\mbox{diag}\left(\ldots,\Delta x,\alpha \Delta x,(1-\alpha)\Delta x, \Delta x,\ldots\right)$,
the first-order DoD-DG operators are periodic upwind SBP operators, i.e.,
$MD^-=-(D^+)^TM$ and $M(D^+-D^-)$ is negative semidefinite.
\begin{proof}
We have
\begin{align*} & MD^-+(D^+)^TM\\
&= \begin{pmatrix}
\ddots & & & &\\-1 & 1 & & &\\ & -1 & 1 & & & \\& & \eta_c-1 & 1-\eta_c & &\\& & -\eta_c & \eta_c-1 & 1 &\\
& & & & -1 & 1 &\\ & & & & & & \ddots
\end{pmatrix}+\begin{pmatrix}
\ddots &  & & & &\\ & -1 & & & & \\ & 1 & -1 & & &\\ & &  1-\eta_c & \eta_c-1 &  &\\[.5mm]
& & \eta_c & 1-\eta_c & -1 & &\\ & & & & 1 & -1 &\\ & & & & & 1 & \ddots
\end{pmatrix} = 0.
\end{align*}
In addition, we obtain
\begin{align*} & M(D^+-D^-)\\
&= \begin{pmatrix}
\ddots &  & & & &\\ & -1 & 1 & & & \\ & & -1 & 1-\eta_c & \eta_c &\\ & & & \eta_c-1 & 1-\eta_c &\\[.5mm]
& & & & -1 & 1 &\\ & & & & & -1 & 1\\ & & & & & & \ddots
\end{pmatrix}-\begin{pmatrix}
\ddots & & & &\\-1 & 1 & & &\\ & -1 & 1 & & & \\& & \eta_c-1 & 1-\eta_c & &\\& & -\eta_c & \eta_c-1 & 1 &\\
& & & & -1 & 1 &\\ & & & & & & \ddots
\end{pmatrix}\\
&=\begin{pmatrix}
\ddots &  & & & &\\ 1 & -2 & 1 & & & \\ & 1 & -2 & 1-\eta_c & \eta_c &\\ & & 1-\eta_c & 2(\eta_c-1) & 1-\eta_c &\\[.5mm]
& &\eta_c & 1-\eta_c & -2 & 1 &\\ & & & & 1 & -2 & 1\\ & & & & & & \ddots
\end{pmatrix}.
\end{align*}
This matrix is symmetric, weakly diagonally dominant with negative diagonal
 entries and is thus negative semidefinite.
\end{proof}
\end{lemma}

\subsection{SBP properties of DoD with central numerical flux}
\label{subsec:SBP_central}
As elaborated in Remark~\ref{rmk:SBP_bilinear_equivalence}, periodic SBP
properties are indicated by skew-symmetry of the bilinear
forms $a_h^\z$ and $J_h^\z$ in the DoD stabilized scheme \eqref{eq:dod_semidiscretization}.
A straightforward computation shows
that $J_h^\z$ is independent of the choice of $L_c$ and $R_c$, as long as
$L_c+R_c=1$.
Skew-symmetry of the bilinear form $a_h^\z$ for central fluxes is well-known
\cite[Section~2.2]{Hesthaven2007}, so it
remains to show skew-symmetry of $J_h^\z(u_h,w_h)$ in case of central fluxes.

\begin{lemma}\label{lemma:central_skew_symmetry}
The DoD stabilization based on central numerical fluxes is skew-symmetric.
\end{lemma}
\begin{proof}
In the case that central numerical fluxes ${\cal H}(a,b)=\frac12(a+b)$ are employed,
the interface stabilizing
terms $J_h^{0, \z}$ given in \eqref{eq:def_J0_bilinar} reduce to
\begin{align*}
\frac{2}{\eta_{c}}J_h^{0, \z}(u_h,w_h) &= (u_{c+1}-u_{c})(x_{c-\frac12})\cdot \llbracket w_h\rrbracket_{x_{c-\frac12}} + (u_{c-1}-u_{c})(x_{c+\frac12})\cdot \llbracket w_h\rrbracket_{x_{c+\frac12}}\\
&= -(f_2g_1)|_{x_{c-\frac12}} + (f_1g_2)|_{x_{c+\frac12}},
\end{align*}
where the abbreviation $f_1=u_{c-1}-u_{c},\ f_2 = u_{c}-u_{c+1},\ g_1=w_{c-1}-w_{c},\ g_2 = w_{c}-w_{c+1}$ has been employed.
By \eqref{eq:def_J1_bilinar}, we can write $J^{1, \z}$ as
\begin{align*}
\frac{1}{\eta_{c}}J_h^{1, \z}(u_h,w_h) &= \int_{E_c}\sum_{j\in{\cal I}_c} K(j)H(j)\,dx\,,
\end{align*}
where $H$ is given by
\begin{align*}
H(j) = \left({\cal H}(u_{c-1} , u_{c+1}) - f(u_j)\right) \partial_x w_j + {\cal H}_a (u_{c-1}, u_{c+1})u_j \partial_x w_{c-1} +{\cal H}_b (u_{c-1}, u_{c+1})u_j \partial_x w_{c+1},
\end{align*}

Inserting central numerical fluxes into the terms $H(j)$ and using $f(u)=u$, we obtain
\begin{align*}
H(j) = \left(\frac12\left(u_{c-1}+u_{c+1}\right) - u_j\right) \partial_x w_j + \frac12 u_j \partial_x w_{c-1} + \frac12 u_j \partial_x w_{c+1},
 \end{align*}
and thus
\begin{align*}
H(c-1) &= \frac12u_{c+1} \partial_x w_{c-1}  + \frac12 u_{c-1} \partial_x w_{c+1},\\
H(c) &=  \left(\frac12\left(u_{c-1}+u_{c+1}\right) - u_{c}\right) \partial_x w_{c} + \frac12 u_{c} \partial_x w_{c-1} + \frac12 u_{c} \partial_x w_{c+1}, \\
H(c+1) &= \frac12 u_{c-1}\partial_x w_{c+1} + \frac12 u_{c+1} \partial_x w_{c-1}\,.
 \end{align*}
Summarizing the terms with $K(c)=-1, \; K(c-1)+K(c+1)=1$ yields
\begin{align*}
\int_{E_c}\sum_{j\in{\cal I}_c} K(j)H(j)\,dx &= \frac12\int_{E_c}\left(u_{c+1}-u_{c}\right)\partial_x \left( w_{c-1}-w_{c} \right)+ \left(u_{c-1}-u_{c}\right) \partial_x \left(w_{c+1}-w_{c}\right)\, dx \\
&= -\frac12\int_{E_c}f_2\partial_x g_1+ f_1 \partial_x g_2\, dx.
\end{align*}

Proceeding to show skew-symmetry, we have
\begin{align*}
\frac{2}{\eta_{c}}\left(J_h^{0, \z}(u_h,w_h) + J_h^{0, \z}(w_h,u_h)\right) &= -(f_2g_1)|_{x_{c-\frac12}} + (f_1g_2)|_{x_{c+\frac12}} -  (g_2f_1)|_{x_{c-\frac12}} + (g_1f_2)|_{x_{c+\frac12}}\\
&= \left[f_1g_2+f_2g_1\right]_{x_{c-\frac12}}^{x_{c+\frac12}}\,.
\end{align*}
Furthermore, integration by parts yields
\begin{align*}
\frac{2}{\eta_{c}}\left(J_h^{1, \z}(u_h,w_h) + J_h^{1, \z}(w_h,u_h)\right) &= -\int_{c}f_2\partial_x g_1 + f_1 \partial_x g_2 + g_2\partial_x f_1+ g_1 \partial_x f_2\, dx \\
&= -\left[f_1g_2+f_2g_1\right]_{x_{c-\frac12}}^{x_{c+\frac12}}\,.
\end{align*}
Consequently,
we have
\[\frac{2}{\eta_{c}}\left(J_h^{0, \z}(u_h,w_h) + J_h^{0, \z}(w_h,u_h)+ J_h^{1, \z}(u_h,w_h) + J_h^{1, \z}(w_h,u_h)\right) = 0\]
and skew-symmetry of $J_h^\z = J_h^{0, \z}+J_h^{1, \z}$ is proven.
\end{proof}
By summing up the background and DoD terms, we obtain a result for the full operator
\begin{corollary}
    The full DoD-stabilized bilinear form
    $\widetilde{a}_h^\z(u_h, w_h) = a_h^\z(u_h, w_h) + J^\z_h(u_h, w_h)$
    is skew-symmetric.
\end{corollary}

\subsection{SBP properties of DoD with upwind numerical fluxes}
Like previously, we can tie the bilinear forms
with upwinding ($a^-_h, J^-_h$) and
downwinding ($a^+_h, J^+_h$) to the periodic upwind SBP properties.
We abbreviate by writing $\widetilde{a}^{\{+, -\}}_h = a_h^{\{+, -\}}+J_h^{\{+, -\}}$.
Then, the periodic upwind SBP properties \eqref{eq:bilinearform_upwind_sbp}
need to be fulfilled by $\widetilde{a}^{\{+, -\}}_h$ instead of $a_h^{\{+, -\}}$.
Note that the upwind SBP property of the background scheme consisting solely of
$a_h^{\{+, -\}}$ is well known and also implied by
Theorem~\ref{theorem:upwind_dod_is_sbp} with $\eta_c=0$.
We first consider a naive approach by employing the choice
$L_c=1, R_c=0$ for $J^-_h$ and $L_c=0, R_c=1$ for $J^+_h$ in \eqref{eq:def_J1_bilinar},
based on their flow direction, as proposed in \cite{may2022dod}.
Employing this stabilization for order $p>0$ will
unfortunately not result in a dual-pair of periodic upwind SBP operators. In particular,
the generalized skew-symmetry fails because of the aforementioned choices of $L_c$
and $R_c$, that introduce
direction-based terms, which can not be compensated.

Therefore, our first step is a slight ``symmetrization`` of this method by instead choosing
$L_c=R_c=1/2$ for both $J^-_h$ and $J^+_h$.
In our tests, this choice still provides stable numerical simulations.
Nevertheless, the resulting scheme is not semidiscretely stable for the telegraph equation and
most importantly, the resulting methods do not end up in periodic upwind SBP operators, but are
linked by the central scheme from Section~\ref{subsec:SBP_central} in the following representation.
\begin{lemma}\label{lemma:central_diss_splitting}
Consider the bilinear forms $a^{\{+, -\}}_h, J^{\{+, -\}}_h$ with $L_c=R_c=1/2$
and $a^\z_h, J^\z_h$ from \eqref{eq:dod_semidiscretization}.
Then, there is a dissipation operator $\widetilde{a}_h^{\text{diss}}(u_h, w_h)$, such that
\begin{align*}
    \widetilde{a}^+_h(u_h, w_h) &= \widetilde{a}_h^{\z}(u_h, w_h) - \widetilde{a}_h^{\text{diss}}(u_h, w_h), \\
    \widetilde{a}^-_h(u_h, w_h) &= \widetilde{a}_h^{\z}(u_h, w_h) + \widetilde{a}_h^{\text{diss}}(u_h, w_h).
\end{align*}
\end{lemma}
\begin{proof}
First, we split the dissipation operator in the natural sum of the background
method and the stabilization part by writing
$\widetilde{a}_h^{\text{diss}}(u_h, w_h)= a_h^{\text{diss}}(u_h, w_h) + J_h^{\text{diss}}(u_h, w_h)$.
We split the operators $\widetilde{a}^+_h(u_h, w_h)$ and $\widetilde{a}^-_h(u_h, w_h)$ in its summands
parts of the background scheme $a^{\{+, -\}}_h(u_h, w_h)$, the flux stabilization $J^{0,\{+, -\}}_h(u_h, w_h)$ and the volume
stabilization $J^{1,\{+, -\}}_h(u_h, w_h)$ and consider them separately.

First, we consider the \underline{background} part.
        We have
        \begin{multline*}
            a_h^{\delta}(u_h, w_h)=-\sum\limits_{i=1}^N\int\limits_{x_{i-1/2}}^{x_{i+1/2}}u_h\partial_xw_h\text{d}x
            \\
            + \sum\limits_{i=1}^N\frac{1}{2}(u_i+u_{i+1})(w_i-w_{i+1})(x_{i+\frac{1}{2}}) + \frac{1}{2}(w_i-w_{i+1})
            \begin{cases}
                (u_i-u_{i+1})(x_{i+\frac{1}{2}}), \quad \vel>0 \ (\delta = -)\\
                (u_{i+1}-u_{i})(x_{i+\frac{1}{2}}), \quad \vel<0\ (\delta = +).
            \end{cases}
        \end{multline*}
        We obtain immediately the desired splitting by choosing
        \begin{align*}
            a_h^{\text{\z}}(u_h, w_h)&= -\sum\limits_{i=1}^N\int\limits_{x_{i-1/2}}^{x_{i+1/2}}u_h\partial_xw_h\text{d}x
                + \sum\limits_{i=1}^N\frac{1}{2}(u_i+u_{i+1})(w_i-w_{i+1})(x_{i+\frac{1}{2}}),\\
            a_h^{\text{diss}}(u_h, w_h)&= \frac{1}{2}(u_i-u_{i+1})(w_i-w_{i+1})(x_{i+\frac{1}{2}}).
        \end{align*}

Next, we consider the \underline{flux stabilization}.
The central part of $J_h^{0,\{+, -\}}$ is given by
        \begin{equation*}
            J_h^{0,\z}(u_h, w_h)= \frac{\eta_c}{2}\left((u_{c+1}-u_c)(w_{c-1}-w_c)(x_{c-\frac{1}{2}})+(u_{c-1}-u_c)(w_c-w_{c+1})(x_{c+\frac{1}{2}})\right).
        \end{equation*}
        Therefore, we can write the upwind stabilizations as
        \begin{align*}
            J_h^{0,+}(u_h, w_h) &=\eta_c (u_{c+1}-u_c)(w_{c-1}-w_c)(x_{c-\frac{1}{2}})\\
                                &= J_h^{0,c}(u_h, w_h) + \frac{\eta_c}{2}\left((u_{c+1}-u_c)(w_{c-1}-w_c)(x_{c-\frac{1}{2}})-(u_{c-1}-u_c)(w_c-w_{c+1})(x_{c+\frac{1}{2}})\right),\\
            J_h^{0,-}(u_h, w_h) &= \eta_c (u_{c-1}-u_c)(w_c-w_{c+1})(x_{c+\frac{1}{2}})\\
                                &= J_h^{0,c}(u_h, w_h) + \frac{\eta_c}{2}\left((u_{c-1}-u_c)(w_c-w_{c+1})(x_{c+\frac{1}{2}}) - (u_{c+1}-u_c)(w_{c-1}-w_c)(x_{c-\frac{1}{2}})\right),
        \end{align*}
        and obtain the desired splitting by the choice
        $$J_h^{0, \text{diss}}(u_h, w_h) := \frac{\eta_c}{2}\left((u_{c-1}-u_c)(w_c-w_{c+1})(x_{c+\frac{1}{2}}) - (u_{c+1}-u_c)(w_{c-1}-w_c)(x_{c-\frac{1}{2}})\right).$$

Finally, we consider the \underline{volume stabilization}.
The central part of $J_h^{1,\{+, -\}}$ is given by
        \begin{equation*}
            J_h^{1,\z}(u_h, w_h)= \frac{\eta_c}{2}\int\limits_{x_{c-\frac{1}{2}}}^{x_{c+\frac{1}{2}}}
                 (u_{c+1}-u_c)\partial_x(w_{c-1}-w_c) + (u_{c-1}-u_c)\partial_x (w_{c+1}-w_c)\text{d}x.
        \end{equation*}
        Now we calculate the upwind volume fluxes by \cite[eq. (10)]{may2022dod} with $L=R=1/2$ to receive
        \begin{align*}
            J_h^{1,-}(u_h, w_h) &= \eta_c\int\limits_{x_{c-\frac{1}{2}}}^{x_{c+\frac{1}{2}}}
            -(u_{c-1}-u_c)\partial_x w_c + \frac{1}{2}(u_{c-1}-u_{c+1})\partial_x w_{c+1} + \frac{1}{2}(u_{c-1}-2u_c+u_{c+1})\partial_x w_{c-1}
            \text{d}x,\\
        J_h^{1,+}(u_h, w_h) &= \eta_c\int\limits_{x_{c-\frac{1}{2}}}^{x_{c+\frac{1}{2}}}
        -(u_{c+1}-u_c)\partial_x w_c + \frac{1}{2}(u_{c+1}-u_{c-1})\partial_x w_{c-1} + \frac{1}{2}(u_{c-1}-2u_c+u_{c+1})\partial_x w_{c+1}
            \text{d}x.
        \end{align*}
        Again by simply rearranging the terms to isolate $J_h^{1, \z}$ results in choosing the following dissipation term
        \begin{equation*}
            J_h^{1, \text{diss}}(u_h, w_h) := \eta_c\int\limits_{x_{c-\frac{1}{2}}}^{x_{c+\frac{1}{2}}}
            (u_{c+1}-u_{c-1})\partial_x w_c + (u_c-u_{c+1})\partial_x w_{c+1} + (u_{c-1}-u_c)\partial_x v_{c-1}
            \text{d}x,
        \end{equation*}
        that lets us write
        \begin{align*}
            J_h^{1,+}(u_h, w_h) &= J_h^{1,\z}(u_h, w_h) - J_h^{1, \text{diss}}(u_h, w_h)\\
            J_h^{1,-}(u_h, w_h) &= J_h^{1,\z}(u_h, w_h) + J_h^{1, \text{diss}}(u_h, w_h).
        \end{align*}
        We can combine now
        \begin{equation}\label{eq:dissipation_operator}
        \widetilde{a}_h^{diss}(u_h, w_h) = {a}_h^{diss}(u_h, w_h) + J_h^{0, \text{diss}}(u_h, w_h) + J_h^{1, \text{diss}}(u_h, w_h)
        \end{equation}
        that proves the statement.
\end{proof}

To obtain a periodic upwind SBP pair, we would require the dissipation term
$\widetilde{a}_h^{diss}$ to be symmetric. Note that the dissipation term of
the background method $a_h^{\text{diss}}(u_h, w_h)$ is already symmetric.
So our strategy is to split the
dissipation part of the stabilization and apply one half to stabilize $1/2 \ a_h^{\text{diss}}(u_h, w_h)$
as previously and the other half to stabilize $1/2 \ a_h^{\text{diss}}(w_h, u_h) = 1/2 \  a_h^{\text{diss}}(u_h, w_h)$
by $1/2 \ J_h^{\text{diss}}(w_h, u_h)$. This is therefore a consistent modification of the DoD-scheme. In other words,
we symmetrize the dissipation term. The desired properties are stated in the following theorem.
\begin{theorem}\label{theorem:upwind_dod_is_sbp}
    Consider the modified upwind schemes
    \begin{align*}
        \widetilde{a}_h^{+,\text{symm}}(u_h, w_h) &:= \widetilde{a}_h^{\z}(u_h, w_h) -
        1/2\left(\widetilde{a}_h^{\text{diss}}(u_h, w_h) + \widetilde{a}_h^{\text{diss}}(w_h, u_h)\right),\\
        \widetilde{a}_h^{-,\text{symm}}(u_h, w_h) &:= \widetilde{a}_h^{\z}(u_h, w_h) +
        1/2\left(\widetilde{a}_h^{\text{diss}}(u_h, w_h) + \widetilde{a}_h^{\text{diss}}(w_h, u_h)\right),
    \end{align*}
    where $\widetilde{a}_h^{\text{diss}}$ is defined as in Lemma~\ref{lemma:central_diss_splitting}, Eq.~\eqref{eq:dissipation_operator}. Then,
    $\left(\widetilde{a}_h^{+,\text{symm}}(u_h, w_h), \widetilde{a}_h^{-,\text{symm}}(u_h, w_h)\right)$ is
    a dual-pair of periodic upwind SBP operators, i.e.,
    \begin{itemize}
        \item[a)] $\widetilde{a}^{+,\text{symm}}_h(u_h, w_h) = -\widetilde{a}^{-,\text{symm}}_h(w_h, u_h)$,
        \item[b)] $\widetilde{a}^{+,\text{symm}}_h(u_h, u_h) - \widetilde{a}^{-,\text{symm}}_h(u_h, u_h) \le 0$.
    \end{itemize}
\end{theorem}
    \begin{proof}
        Using the skew symmetry of $\widetilde{a}_h^{\z}$ (obtained by Lemma~\ref{lemma:central_skew_symmetry}),
        a) is immediately fulfilled by direct computation.

        For b), we receive
        \begin{align*}
           \widetilde{a}^{+,\text{symm}}_h(u_h, u_h) - \widetilde{a}^{-,\text{symm}}_h(u_h, u_h)
            = - 2\widetilde{a}_h^{\text{diss}}(u_h, u_h).
        \end{align*}
        It remains to show that $\widetilde{a}_h^{\text{diss}}(u_h, u_h)\ge0$. In \cite[Theorem 6]{may2022dod}
        is proven, that
        \begin{equation}\label{eq:positivity_of_dod_bilinearform}
            \widetilde{a}_h^{-}(u_h, u_h) \ge 0 \quad \text{for any} \; u_h.
        \end{equation}
        Remark that we specified here the special case of the transport equation with an upwind flux and $L_c=R_c=1/2$. Further, the
        skew-symmetry of $\widetilde{a}_h^{\z}$ implies
        \begin{equation*}
            \widetilde{a}_h^{\z}(u_h, u_h)=0.
        \end{equation*}
        Expanding $\widetilde{a}_h^{\text{-}}$ and using the previous identity, we conclude by achieving
        \begin{equation*}
            0 \le \widetilde{a}_h^{\text{-}}(u_h, u_h) = \widetilde{a}_h^{\z}(u_h, u_h) + \widetilde{a}_h^{\text{diss}}(u_h, u_h)
              = \widetilde{a}_h^{\text{diss}}(u_h, u_h).
              \qedhere
        \end{equation*}
    \end{proof}

\subsection{Summary of SBP properties}
Let the derivative operators $D^\z, D^+,D^-$ be defined by
\begin{equation}\label{eq:D1_D2_pairs_dod}
\left(a^{\z}_h+J^\z_h\right)(u_h, w_h) = \underline{u}^TD^\z\underline{w}, \quad \widetilde{a}^{+, \text{symm}}_h(u_h, w_h)
= \underline{u}^TD^+\underline{w}, \quad \widetilde{a}^{-, \text{symm}}_h(u_h, w_h) = \underline{u}^TD^-\underline{w}.
\end{equation}
Then, $D^\z$ is a periodic SBP operator in the classical sense and the pair
$(D^+,D^-)$ is a dual-pair of upwind SBP operators, which can be decomposed in
a central part given by $D^\z$ and a dissipation part.
Therefore, these operators are included in the analysis of Section~\ref{sec:analysis}.

Further, the classic DoD scheme \eqref{eq:def_J0_bilinar} already inherits periodic
upwind SBP properties for $p=0$, as shown in Section~\ref{subsec:SBP_dod_p=0}.
Therefore, the symmetrization procedure is not necessary in this special case.

\begin{remark}
    The (upwind) SBP properties derived in this chapter transfer
    directly to the DGSEM with Gauss-Lobatto-Legendre collocation
    nodes \cite[Section~3.2]{Kopriva2009}, as all involved integrals
    (including the derivations in \cite[Theorem 6]{may2022dod} to prove
    \eqref{eq:positivity_of_dod_bilinearform}) contain
    polynomials of maximum degree $2p-1$ and are therefore exact using
    GLL quadrature.
\end{remark}

\section{Numerical results}
\label{sec:numerics}
In the following, we provide numerical results for the DoD-stabilized schemes defined by
the operators \eqref{eq:D1_D2_pairs_dod} resulting in the semidiscretization
\eqref{eq:semidiscretization_bilinearform} or \eqref{eq:heat_semidiscretization}
respectively, to underline our theoretical findings.
This includes a convergence test, followed by a numerical asymptotic analysis
for $\epsilon \rightarrow 0$. Finally, we examine
the relevance of DoD by treating the asymptotic limit equation implicitly in time, as
this is closer to a practical application. We only show results for one of the two
alternating flux cases respectively, i.e., $\Drho = D^-\;, \Dg=D^+$, but we want to
emphasize that the reversed case did not lead to any remarkable discrepancies.
All the numerical methods are implemented in Julia \cite{bezanson2017julia}, while the
code is based on the infrastructure of
SummationByPartsOperators.jl~\cite{ranocha2021sbp}. The plots are generated
with Makie.jl~\cite{danisch2021makie}.
All code and data required to reproduce the numerical results are
available in our reproducibility repository \cite{petri2025kineticRepro}.

\subsection{Convergence}
In the following, we want to present numerical convergence results by applying the
DoD-stabilized operators \eqref{eq:D1_D2_pairs_dod} in the scheme
\eqref{eq:semidiscretization_matrixform} with an
appropriate IMEX splitting \eqref{eq:imex_splitting} in time to the telegraph equation on a cut-cell
mesh with five small cut cells of size $\alpha \Delta x$, where $\Delta x$ is the
background mesh size and $\alpha \in \{10^{-7}, 10^{-3}, 10^{-1}, 0.3, 0.49\}$.
We examine different values of $\epsilon$ and apply the following condition to the
time step size and the penalty parameter:
$$\Delta t = \frac{C_{\text{pre}}(p)}{2p+1}\epsilon \Delta x, \quad \eta_c = 1-\frac{\alpha}{\lambda_c(p)}, \;
\text{where }
\begin{cases}\lambda_c(p) = 1.0, \; C_{\text{pre}}(p) = 0.5,  \quad &p=0 \\
    \lambda_c(p) = 0.55, \; C_{\text{pre}}(p) = 0.3, \quad &p=1 \\
    \lambda_c(p) = 0.45, \; C_{\text{pre}}(p) = 0.15, \quad &p=2 \\
\end{cases}.$$
The reduced stabilization and CFL condition by increasing the degree of basis
functions are necessary to ensure stability, see \cite{petri2025domain}.
For time-stepping, we use the third-order IMEX method ARS(4,4,3) \cite[Sec.~2.8]{ascher1997imex}, given by the following Butcher tableaux:
\begin{equation}\label{eq:ARS3_tableau}
\begin{array}{c|ccccc}
0 & 0 & 0 & 0 & 0 & 0\\
1/2 & 1/2 & 0 & 0 & 0 & 0\\
2/3 & 11/18 & 1/18 & 0 & 0 & 0\\
1/2 & 5/6 & -5/6 & 1/2 & 0 & 0\\
1 & 1/4 & 7/4 & 3/4 & -7/4 & 0\\
\hline
 & 1/4 & 7/4 & 3/4 & -7/4 & 0
\end{array}
\qquad
\begin{array}{c|ccccc}
0 & 0 & 0 & 0 & 0 & 0\\
1/2 & 0 & 1/2 & 0 & 0 & 0\\
2/3 & 0 & 1/6 & 1/2 & 0 & 0\\
1/2 & 0 & -1/2 & 1/2 & 1/2 & 0\\
1 & 0 & 3/2 & -3/2 & 1/2 & 1/2\\
\hline
 & 0 & 3/2 & -3/2 & 1/2 & 1/2
\end{array}
\end{equation}
Convergence results for the first component are given in Figure~\ref{fig:convergence_1D_telegraph}
for basis polynomials of degrees zero to two. We compare our scheme with the exact solution
\cite[section 4.1]{Jang_etal:2014}
\begin{equation*}
  \begin{cases}
    \rho(x, t) = \frac{1}{r}e^{rt}\sin(x), \quad r = \frac{-2}{1+\sqrt{1-4\epsilon^2}},\\
    \tg(x, t) = e^{rt}\cos(x),
  \end{cases}
\end{equation*}
on $\Omega_x = [-\pi, \pi]$.
\begin{figure}
  \begin{minipage}[t]{0.32\textwidth}
    \centering
    \includegraphics[width=\textwidth]{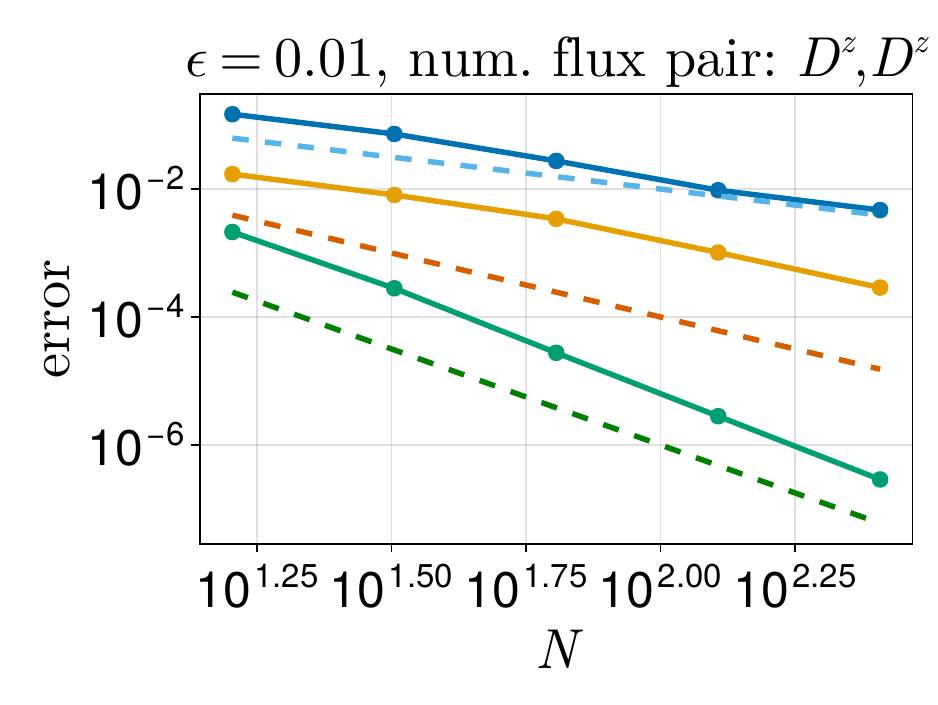}\\
  \end{minipage}
  \begin{minipage}[t]{0.32\textwidth}
    \centering
    \includegraphics[width=\textwidth]{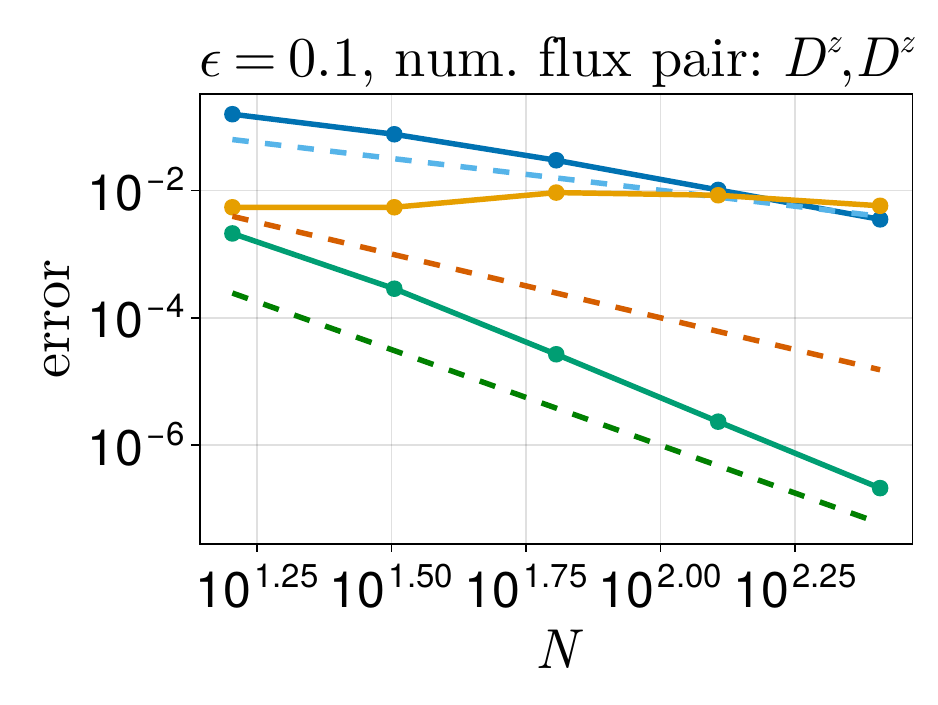}\\
  \end{minipage}
  \begin{minipage}[t]{0.32\textwidth}
    \centering
    \includegraphics[width=\textwidth]{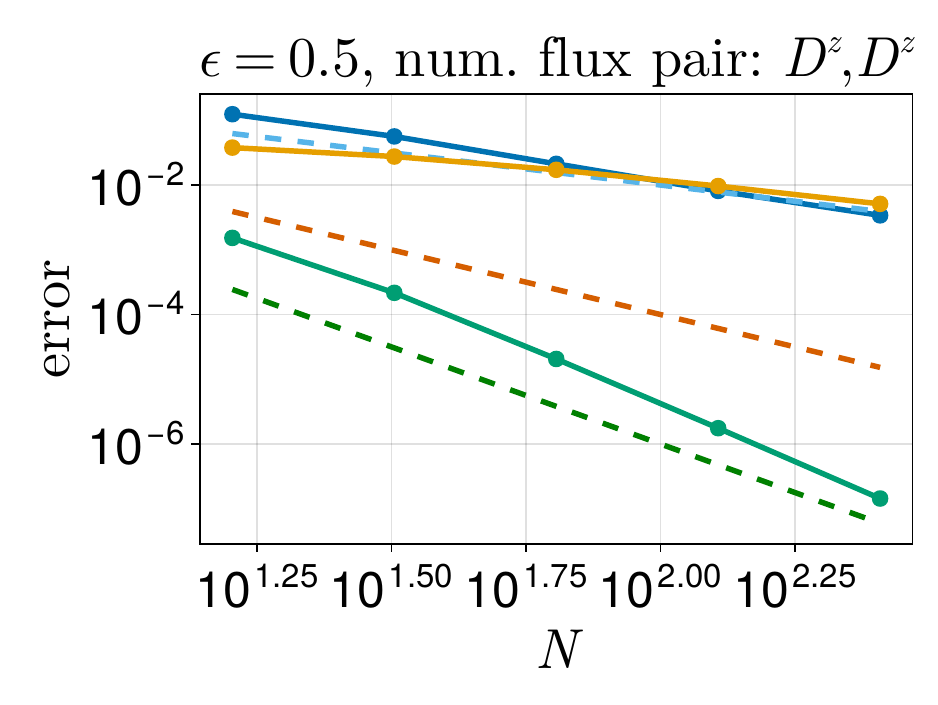}\\
  \end{minipage}\\

  \begin{minipage}[t]{0.32\textwidth}
    \centering
    \includegraphics[width=\textwidth]{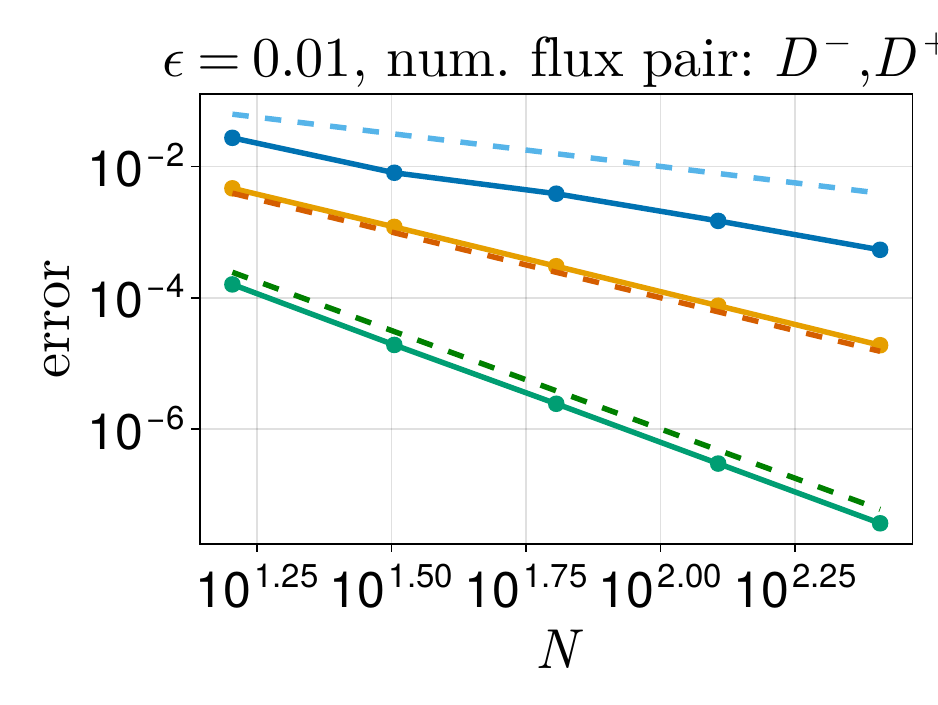}\\
  \end{minipage}
  \begin{minipage}[t]{0.32\textwidth}
    \centering
    \includegraphics[width=\textwidth]{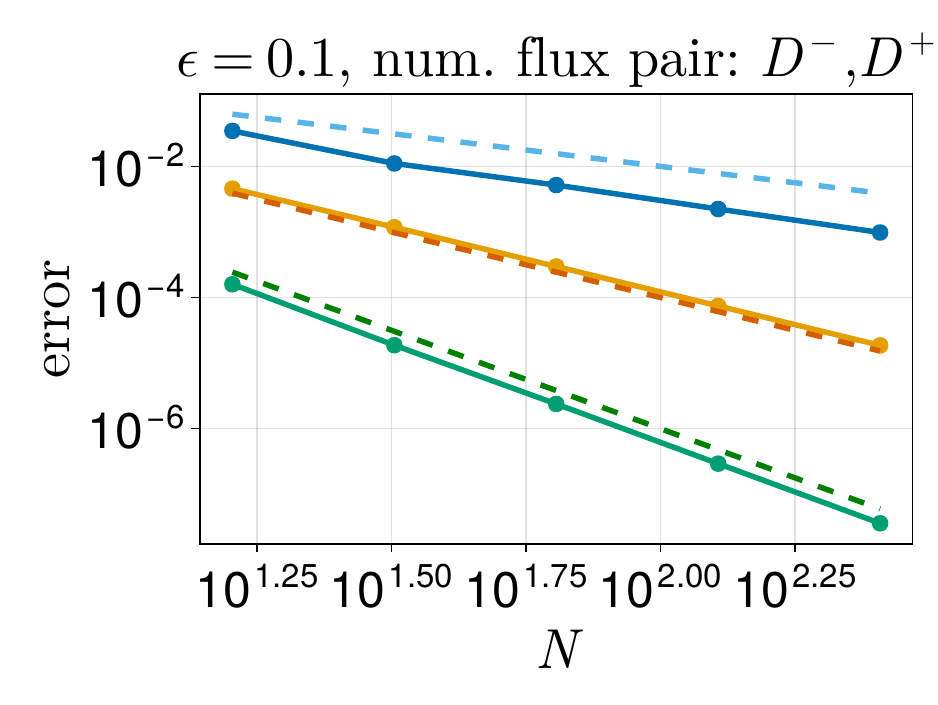}\\
  \end{minipage}
  \begin{minipage}[t]{0.32\textwidth}
    \centering
    \includegraphics[width=\textwidth]{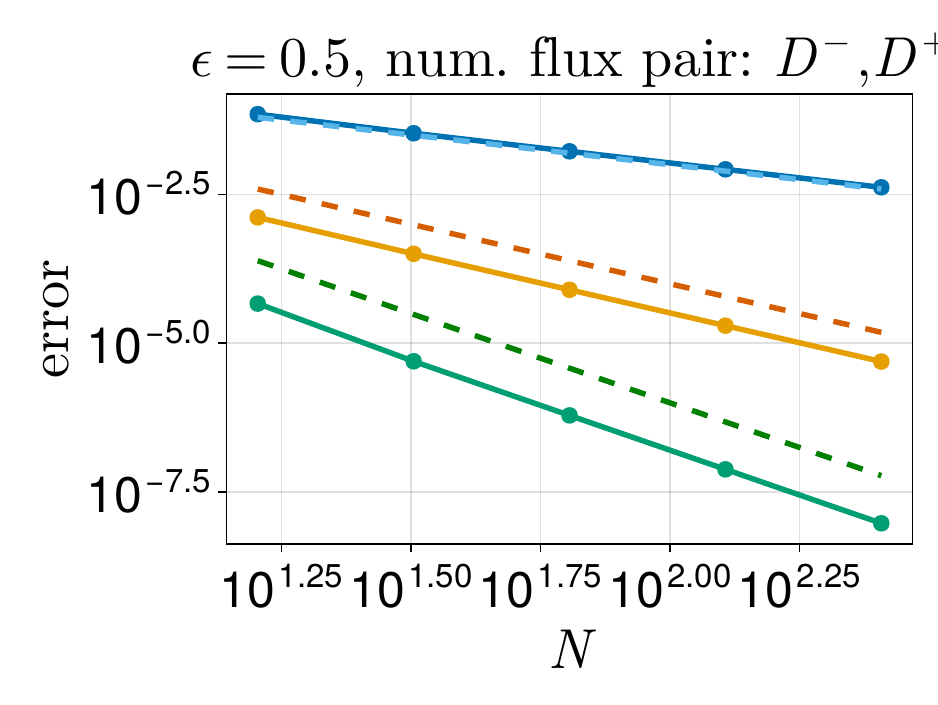}\\
  \end{minipage}
  \centering
  \includegraphics[width=0.7\textwidth, trim ={0 5.2cm 0 5cm} , clip]{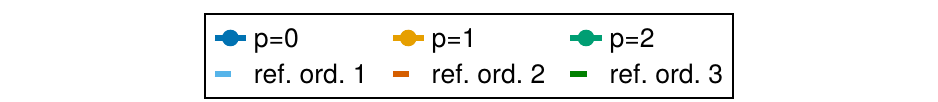}
  \caption{$L^2$ errors for the DoD-stabilized telegraph equation at the time $T=1$.}
  \label{fig:convergence_1D_telegraph}
\end{figure}
In the case of the alternating numerical fluxes, we observe an experimental
order of convergence $p + 1$ as expected. For the central numerical fluxes, we
observe reduced convergence for $p=1$, which aligns with the results of the
background method in \cite{Jang_etal:2014}.
We also observe, that this effect is less severe for smaller $\epsilon$.
\begin{figure}
  \begin{minipage}[t]{0.32\textwidth}
    \centering
    \includegraphics[width=\textwidth]{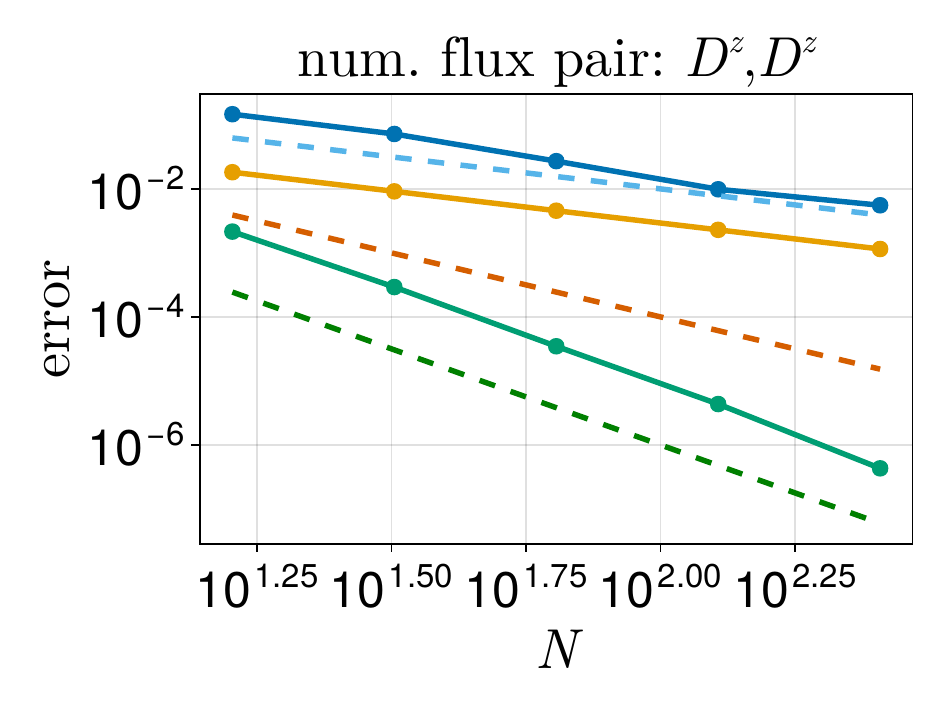}\\
  \end{minipage}
  \begin{minipage}[t]{0.32\textwidth}
    \centering
    \includegraphics[width=\textwidth]{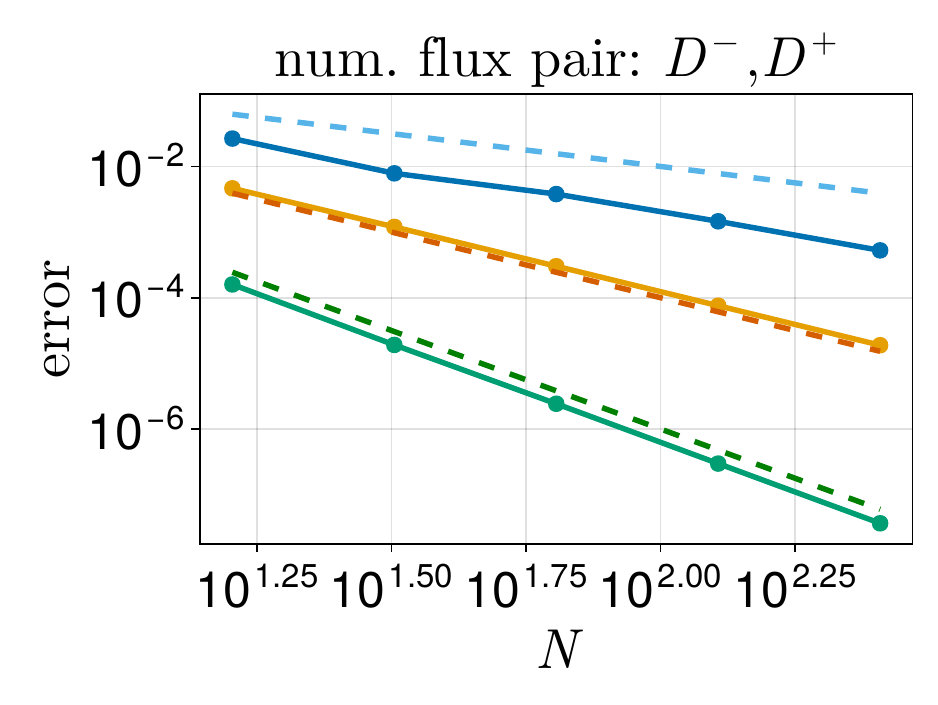}\\
  \end{minipage}
  \centering
  \caption{$L^2$ errors for the DoD-stabilized heat equation at the time $T=1$.}
  \label{fig:convergence_1D_heat}
\end{figure}

In Figure~\ref{fig:convergence_1D_heat}, we show convergence results for the
heat equation, treated with the explicit part of the ARS(4,4,3) method,
as the formal asymptotic analysis leads to this discretization in the limit
$\epsilon\rightarrow 0$. The remaining setup and parameters
do not change in comparison to the setting for the telegraph equation
except for choosing $C_{\text{pre}}(2)=0.0375$ to obtain
stable numerical results.
Again, we receive the expected
convergence results for the alternating fluxes, while the central fluxes
seem to perform better for the heat equation than for the telegraph equation.
Similar to the telegraph equation, the optimal orders
of $p+1$ are not reached exactly for central fluxes and $p=1$.

\subsection{Asymptotic analysis}
To numerically verify the asymptotic convergence of the cut-cell scheme, we
compare numerical solutions of the telegraph equation and the heat equation
for an aligning setup using a discrete set of $\epsilon$‘s that approach $0$.
We consider a grid of $N=2^4$ background cells and insert cuts, such that
we obtain cut cells with the same fractions $\alpha$ as in the last section.
We also use the same values for $\lambda_c$, depending on the polynomial
degree of basis functions in space and use the parabolic CFL condition for
every calculation, to not be restricted by small $\epsilon$. Recall that
this is expected to also provide stable results for the telegraph case as
mentioned in Section~\ref{subsec:semidiscr_background}.

Using Gauss-Legendre collocation nodes $\vec{x}$, we consider well-prepared
initial conditions
\begin{align*}
    \rho^0 &= \frac{1}{r}\sin(\vec{x}), &\quad r = \frac{-2}{1+\sqrt{1-4\epsilon^2}},\\
    \tg^0 &= -\Dg \rho^0,
\end{align*}
for the telegraph equation and $\rho^0$ for the heat equation. For time-stepping,
we use again the ARS(4,4,3) (\eqref{eq:ARS3_tableau}, GSA Type II) and the
second order IMEX SSP2(3,2,2) \cite{pareschi2005} (non-GSA Type I), given by the Butcher-tableaux
\begin{equation*}
\begin{array}{c|ccc}
0   & 0   & 0   & 0\\
1/2 & 1/2 & 0   & 0\\
1   & 1/2 & 1/2 & 0\\
\hline
    & 1/3 & 1/3 & 1/3
\end{array}
\qquad
\begin{array}{c|ccc}
1/4 & 1/4 & 0   & 0\\
1/4 & 0   & 1/4 & 0\\
1   & 1/3 & 1/3 & 1/3\\
\hline
    & 1/3 & 1/3 & 1/3
\end{array}
\end{equation*}
In Figure~\ref{fig:asymptotic_conv}, we show their asymptotic behavior that align
with the theoretical results. Some implementation notes and details for the ARS(4,4,3)
method are given in Appendix~\ref{sec:appendix}.

\begin{figure}
  \begin{minipage}[t]{0.4\textwidth}
    \centering
    \includegraphics[width=\textwidth]{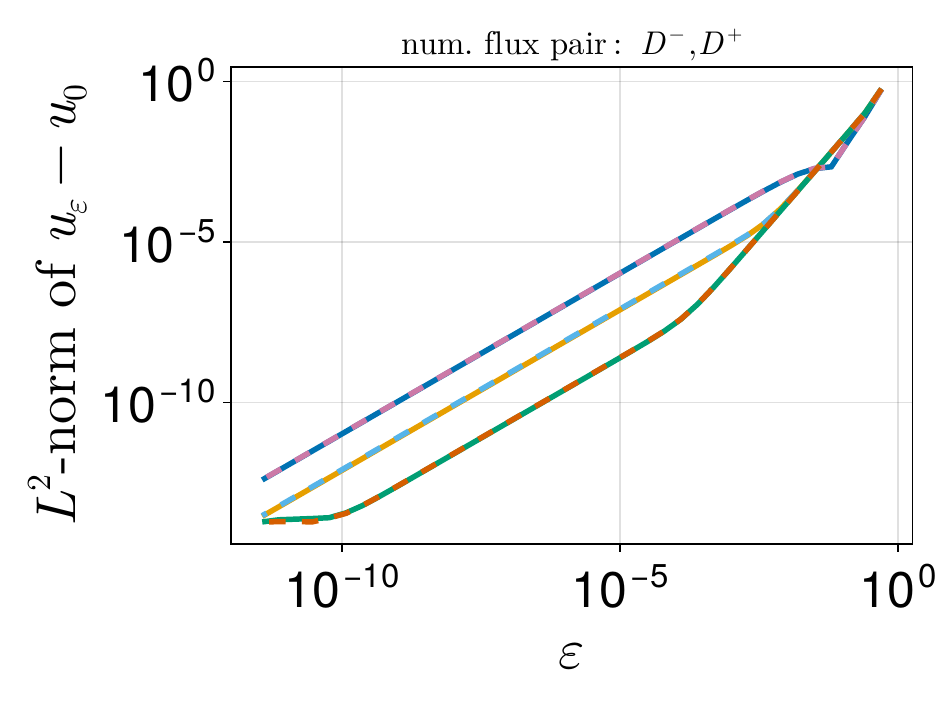}\\
  \end{minipage}
  \begin{minipage}[t]{0.4\textwidth}
    \centering
    \includegraphics[width=\textwidth]{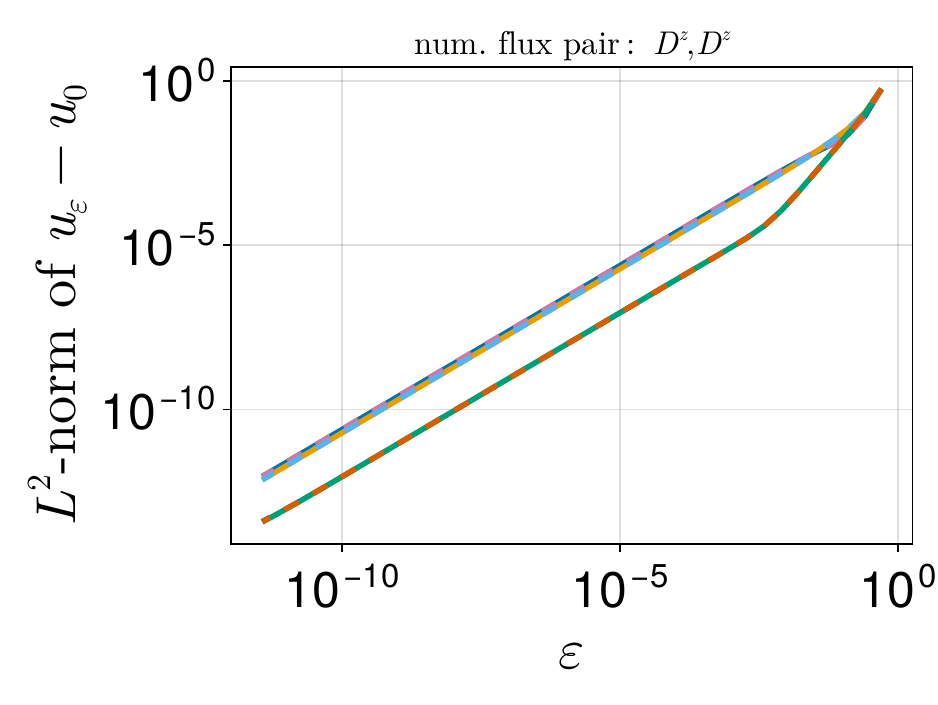}\\
  \end{minipage}
  \centering
  \includegraphics[width=0.7\textwidth, trim ={0 5.2cm 0 5.2cm} , clip]{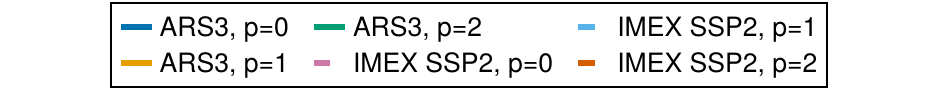}
  \caption{Numerical validation of the asymptotic
    behavior for the DoD-stabilized telegraph
    equation for $\epsilon \to 0$ at the time $T=0.5$. The plots show
    the regularization error measured in the $L^2$-norm for the
    numerical flux pairs $(D^-,D^+)$ and $(D^z,D^z)$, with polynomial degrees $0,1,2$
    and for the two time-stepping methods
    ARS(4,4,3) and IMEX SSP2(3,2,2).}
  \label{fig:asymptotic_conv}
\end{figure}

\subsection{Implicit discretization of the heat equation}
The proposed fully-discrete scheme results in a consistent discretization of the
heat equation, that is explicit in time. Usually, an implicit discretization of
such terms is preferred, to avoid the resulting background CFL restriction
$\Delta t \le C \Delta x^2$ in the explicit case. Therefore, one could ask
how the DoD stabilization impacts the performance of the method under application
of an implicit time discretization: Given the semidiscretization
\eqref{eq:heat_semidiscretization}, the implicit Euler method yields
\begin{equation} \label{eq:impleuler_heat}
    u^{n+1}=\left(\I-\Delta t \Drho \Dg \right)^{-1}u^n.
\end{equation}
The condition of $\left(\I-\Delta t \Drho \Dg \right)$ is crucial for the stability
of the solution to \eqref{eq:impleuler_heat}. Furthermore, solving equations of
the form \eqref{eq:impleuler_heat} also appears as a building block in higher
order implicit RK schemes in every internal stage.
The ($L^2$) operator norm of the matrix $\left(\I-\Delta t \Drho \Dg \right)^{-1}$
scales with $1/\min\limits_i\{\Delta x\}$
without applying a stabilization, which gets resolved to $C_p/\Delta x$ with DoD,
where $C_p$ depends on the polynomial degree $p$ of the basis polynomials in the
employed DG method \cite{petri2025domain}.

In line with this property, we will numerically examine how the condition behaves
when considering cut cells with and without stabilization. For that, we compute
the condition number
\begin{equation}\label{eq:condition_equation}
    \kappa =  \left\|\I-\Delta t \Drho \Dg \right\|_M\left\|\left(\I-\Delta t \Drho \Dg \right)^{-1}\right\|_M,
\end{equation}
where $\|\cdot\|_M$ is the discrete but exact $L^2$ norm. In our experiments,
we use an equidistant grid with $2^7$ background cells in
$\Omega_x = [-\pi, \pi]$ either with or without additional six cuts applied to
the grid.  In the latter case this results in small cut cells of size
$\alpha \Delta x$, where $\Delta x$ is the background mesh size and
$\alpha \in \{10^{-7}, 10^{-3}, 10^{-1}, 0.25, 0.4, 0.49\}$. For the time step,
we chose a parabolic, but cut-cell independent CFL condition
$\Delta t = \Delta x^2/(20(2p+1))$.
The results are displayed
in Table~\ref{table:condition_results}.
\begin{table}[h!]
\caption{Condition numbers \eqref{eq:condition_equation} for different semidiscretizations.}
\label{table:condition_results}
\centering
\begin{tabular}{|c|c|c|c|c|c|c|}
\hline

\multirow{2}{*}{$\mathbf{\kappa}$} &
\multicolumn{2}{|c|}{\textbf{Background}} &
\multicolumn{2}{|c|}{\textbf{Unstabilized}} &
\multicolumn{2}{|c|}{\textbf{DoD (stabilized)}} \\

 & \textbf{$D^-, D^+$} & \textbf{$D^\z, D^\z$}
 & \textbf{$D^-, D^+$} & \textbf{$D^\z, D^\z$}
 & \textbf{$D^-, D^+$} & \textbf{$D^\z, D^\z$} \\
\hline

$p=0$ & \num{1.0318} & \num{1.0080} & \num{7.9578e11} & \num{3.9790e4}  & \num{1.0579} & \num{1.0095} \\
\hline
$p=1$ & \num{1.0955} & \num{1.0424} & \num{3.5257e12} & \num{7.9578e12} & \num{1.3269} & \num{1.0891} \\
\hline
$p=2$ & \num{1.236}  & \num{1.1039} & \num{7.668e12}  & \num{2.8648e12} & \num{2.1555} & \num{1.4514} \\
\hline
$p=3$ & \num{1.4990} & \num{1.2004} & \num{1.4353e13} & \num{5.9011e12} & \num{4.5745} & \num{2.1256} \\
\hline
$p=4$ & \num{1.9242} & \num{1.3424} & \num{2.4982e13} & \num{9.9692e12} & \num{12.0600} & \num{4.9771} \\
\hline
$p=5$ & \num{2.5501} & \num{1.5402} & \num{4.0679e13} & \num{1.5298e13} & \num{31.2335} & \num{9.4726} \\

\hline
\end{tabular}
\end{table}
We observe that the DoD stabilization is crucial to obtain a well-conditioned system
in the presence of small cut cells. While this result verifies numerically that DoD
is well conditioned, it also shows that simply adding implicitness
to the time discretization to avoid a CFL-restriction is not sufficient to solve the small-cell problem.
Using the previously mentioned cut-cell mesh, we can observe these condition numbers numerically
by applying the midpoint rule to the heat equation with $\rho(x,0)=\cos(x)$. We use $N=2^5$ background
cells and $\Delta t = \Delta x/(10(2p+1))$.
In Figure~\ref{fig:heat_simulation_implicit}, we show results for all three different cases displayed in
Table~\ref{table:condition_results} for $t=5$.
\begin{figure}
  \begin{minipage}[t]{0.325\textwidth}
    \centering
    \includegraphics[width=\textwidth]{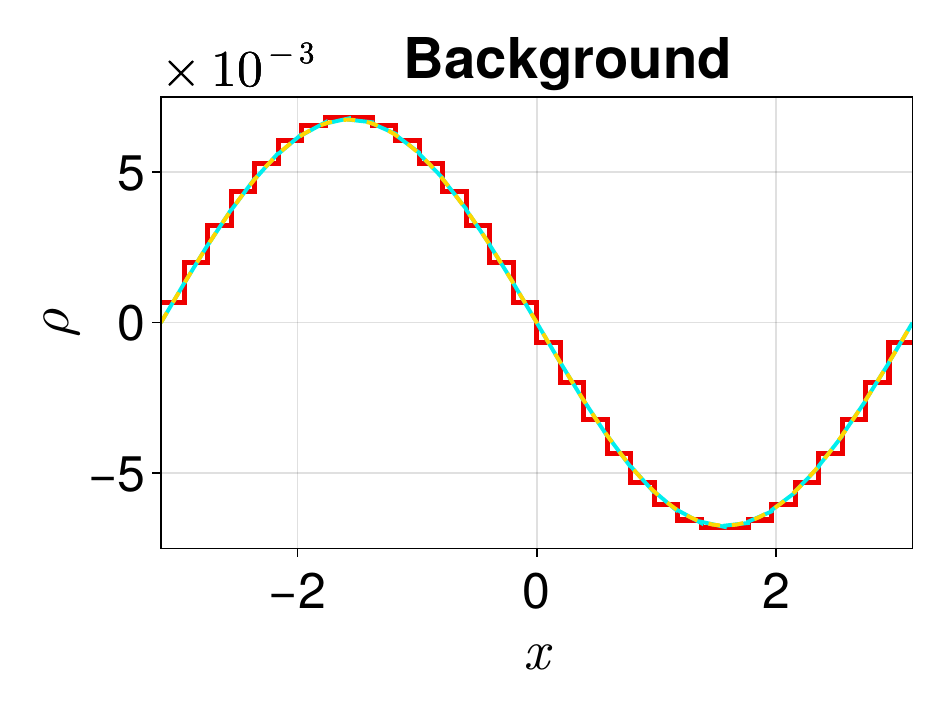}\\
  \end{minipage}
  \begin{minipage}[t]{0.325\textwidth}
    \centering
    \includegraphics[width=\textwidth]{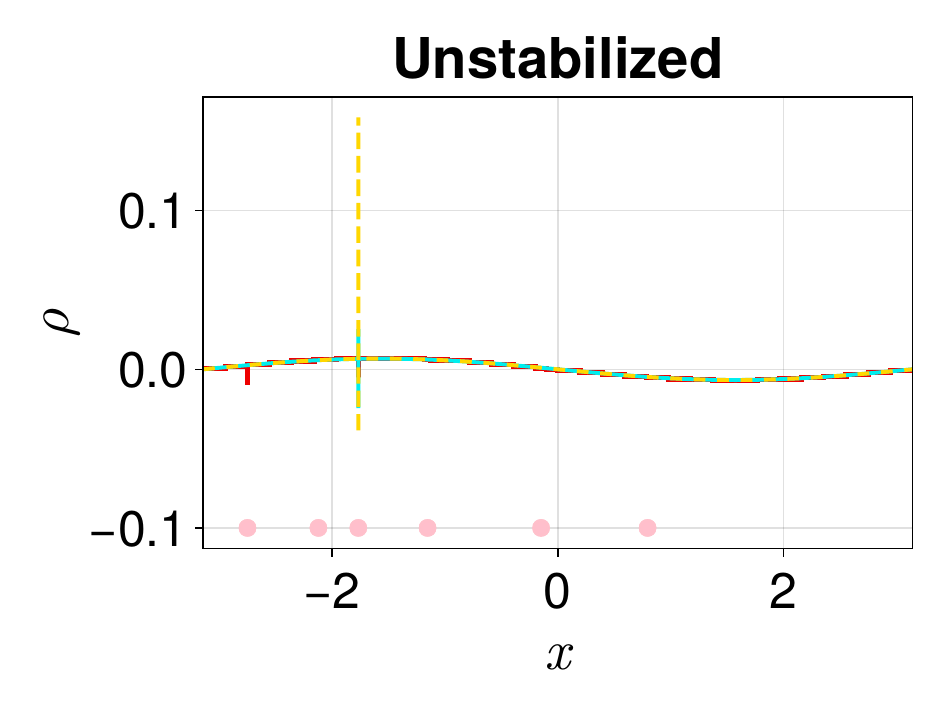}\\
  \end{minipage}
    \begin{minipage}[t]{0.325\textwidth}
    \centering
    \includegraphics[width=\textwidth]{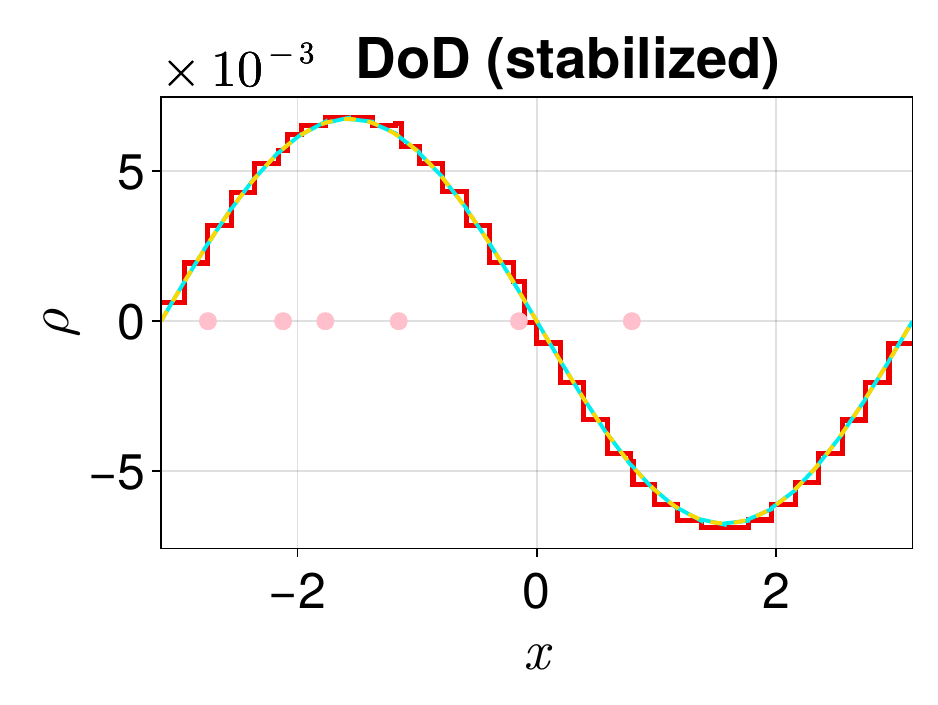}\\
  \end{minipage}
  \centering
  \includegraphics[width=0.7\textwidth, trim ={0 5.2cm 0 5.2cm} , clip]{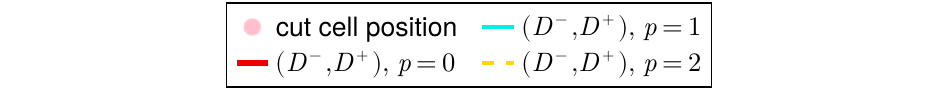}\\
  \caption{Simulation results for the heat equation, using the
    implicit midpoint RK scheme at the time $T=5$.}
  \label{fig:heat_simulation_implicit}
\end{figure}
We observe a behavior hinted at by the significant difference in condition
numbers: The DoD-stabilized scheme provides a stable solution, while the
unstabilized scheme is clearly unstable for every considered case.

Note that the hyperbolic CFL condition used here can also be exchanged by a
parabolic one, without a qualitative change in the solution, as the crucial
differences occur because of the very small cut cells.

\section{Summary and conclusions}
In this article, we introduced cut-cell-stabilized DG methods for
the telegraph equation in one spatial dimension, by extending the
domain-of-dependence stabilization. We proceeded by grasping the
spatial derivatives as linear transport terms and therefore applied
the DoD-terms of the literature to them by respecting the underlying
discrete velocity directions, inherited in this kinetic model. We
proved that including this stabilization also results in
formally asymptotic preserving schemes, using standard IMEX-RK schemes
in time. This also leads to a new cut-cell stabilization term for the
parabolic limit equation, whose advantages we numerically pointed out
even when using implicit time integration.
Further, we showed that using periodic SBP operators lead in general to
AP-explicit schemes using suitable IMEX-RK methods and that
the resulting semidiscretization is stable.
We further adapted the upwind-DoD schemes to obtain DoD-stabilized
periodic upwind SBP operators and proved that using central fluxes and
also upwind fluxes for $p=0$ leads to periodic (upwind) SBP operators.
Finally, we provide numerical results that verify the theoretical findings.
As the SBP property is closely related to entropy stability, we want to
highlight that these results will serve as a foundation to construct
an entropy-stable cut-cell stabilization for DG schemes, which will be
the topic of upcoming investigations.

\label{sec:summary}

\section*{Acknowledgments}

LP and HR were supported by the Deutsche Forschungsgemeinschaft
(DFG, German Research Foundation) under project number 528753982,
within the DFG priority program SPP~2410
(project number 526031774),
and the Daimler und Benz Stiftung (Daimler and Benz foundation,
project number 32-10/22).
GB and CE were supported by the Deutsche Forschungsgemeinschaft (DFG, German Research Foundation) within the DFG priority program SPP~2410
(project number 526031774) and under Germany's Excellence Strategy EXC
2044-390685587, and EXC 2044/2-390685587, Mathematics Münster: Dynamics–Geometry–Structure.
SO acknowledges support by the Deutsche Forschungsgemeinschaft (DFG, German Research Foundation) within the DFG priority program SPP~2410 (project number 526073189).

\printbibliography

\appendix
\section{Stable implementation of ARS schemes}
\label{sec:appendix}
In the following, we elaborate on implementation details of the IMEX Type II schemes.
Performing an asymptotic analysis numerically may require a careful implementation,
because the stiff part can lead to round-off errors, due to its magnitude. Using
64-bit floating point numbers (\texttt{Float64} in Julia) in our simulations, we could also observe that phenomenon for $\epsilon<10^{-7}$.
For implicit RK methods, an optimized implementation is given by
\cite[Chapter IV.8]{hairer1996solving}. This technique can be adapted for IMEX methods,
which provides better results, i.e., the expected asymptotic convergence is seen for
$\epsilon$ down to $10^{-10}$. Nevertheless, the round-off errors start to increase
again at that point. The reason is that not only the stiff part contributes to
large quantities, but also the explicit part by the coefficient of the term $b_h$ in
\eqref{eq:semidiscretization_bilinearform}, independent of cut cells and stabilization.
To resolve this issue, we focus on computing the internal stages of the method such that
no division by $\epsilon$ is required. As we just observed this for the Type II
schemes, we present our implementation specifically for the ARS schemes applied to
\eqref{eq:imex_telegraph_scheme}.
We denote $$\ds{k,l} = \prod\limits_{j=k}^{l} \left(\epsilon^2+\Delta t a_{jj}\right).$$
As $\tilde{a}_{1i}=a_{1i}=0, \; i=1, \hdots s$, we have $\rho^{(1)}=\rho^n, \; \tg^{(1)}=\tg^{n}$.
The internal stages for $\tg^{(k)}$ can be computed explicitly, by using the precomputed, explicit value
$\rho^{(k)}$ and solving for $\tg^{(k)}$, such that
\begin{align*}
    \tg^{(k)} &=\frac{1}{1+\frac{\Delta ta_{kk}}{\epsilon^2}}\left(\tg^n+\frac{\Delta t}{2\epsilon}\sum\limits_{i=1}^{k-1}\widetilde{a}_{ki}\left(D^+-D^-\right)\tg^{(i)}
                -\frac{\Delta t}{\epsilon^2}\sum\limits_{i=1}^{k-1}a_{ki}\left(\Dg\rho^{(i)}+\tg^{(i)}\right)\right).
\end{align*}
By factorizing an additional term $1/\epsilon^2$, we obtain a remaining stage update that does introduce
large quantities by the factor $1/\ds{k,k}$. As this also influences the internal stages of $\rho$, we introduce the more
error-resistant variables
\begin{equation*}
    \hat{\rho}^{(k)} = \ds{2, k-1}\rho^{(k)}, \;\; \hat{g}^{(k)} = \ds{2,k}\tg^{(k)},\quad k = 2, \hdots, s.
\end{equation*}
By also factorizing the respective $1/\ds{k,l}$, we can then rewrite \eqref{eq:imex_telegraph_scheme}
and fully avoid introducing large magnitudes while computing the internal stages by solving the system for $\hat{\rho}^{(k)}, \; \hat{g}^{(k)}$, which is given by
\begin{align*}
    \hat{\rho}^{(k)}  &= \left(\ds{2, k-1}\rho^n -\Delta t \Drho\sum\limits_{i=1}^{k-1}\widetilde{a}_{ki}\ds{i+1, k-1}\hat{g}^{(i)}\right), \\
    \hat{g}^{(k)}  &= \left(\epsilon^2\ds{2, k-1}\tg^{n}+\Delta t \left(\frac{\epsilon}{2}\left(D^+-D^-\right)\left(\sum\limits_{i=1}^{k-1}\widetilde{a}_{ki}\ds{i+1, k-1}\hat{g}^{(i)}\right)
                    - \Dg \sum\limits_{i=1}^{k}a_{ki}\ds{i, k-1}\hat{\rho}^{(i)} - \sum\limits_{i=1}^{k-1}a_{ki}\ds{i+1, k-1}\hat{g}^{(i)}\right)\right), \\
    \rho^{n+1}&=\rho^{(s)}=\frac{1}{\ds{2,s-1}}\hat{\rho}^{(s)},\\
    \tg^{n+1}&=\tg^{(s)}=\frac{1}{\ds{2,s}}\hat{g}^{(s)}.
\end{align*}
This method of computation is used to obtain the results for the ARS(4,4,3) scheme in Figure~\ref{fig:asymptotic_conv}.

\end{document}